\documentclass[a4paper,12pt,reqno]{amsart}
\usepackage[usenames,dvipsnames]{color}

\usepackage{mathtools}
\usepackage{mymacros,fullpage,pstricks}
\usepackage[all]{xy}
\usepackage[
bookmarks=true,bookmarksnumbered=true,
 setpagesize=false,%
 colorlinks=true,%
 pdftitle={},%
 pdfsubject={},%
 pdfauthor={},%
 pdfkeywords={TeX; hyperref; color;},%
 colorlinks=false]{hyperref}

\newcommand*{\myov}[1]{\overbracket[0.65pt][-1pt]{#1}}

\newcommand{\Out}{\operatorname{Out}}
\newcommand{\Bir}{\operatorname{Bir}}

\renewcommand{\Dbar}{\overline{D}}
\renewcommand{\Dhat}{\widehat{D}}
\newcommand{\mult}{\operatorname{mult}}

\newcommand{\bdiscrep}{\operatorname{b-discrep}}

\newcommand{\Br}{\operatorname{Br}}

\newcommand{\bk}{\bsk}
\newcommand{\Sch}{\frakS\mathrm{ch}}
\newcommand{\Div}{\operatorname{Div}}
\newcommand{\bDiv}{\operatorname{\mathbf{Div}}}
\newcommand{\Ab}{\operatorname{\mathbf{Ab}}}
\newcommand{\bdiv}{\operatorname{\mathbf{div}}}
\newcommand{\divi}{\operatorname{div}}
\newcommand{\Exc}{\operatorname{Exc}}
\newcommand{\Bl}{\operatorname{\mathbf{Bl}}}
\newcommand{\Gm}{\mathbb{G}_m}
\newcommand{\R}{\mathbb{R}}
\newcommand{\Q}{\mathbb{Q}}
\newcommand{\Z}{\mathbb{Z}}

\newcommand{\trdeg}{\operatorname{trdeg}}
\newcommand{\Gal}{\operatorname{Gal}}
\newcommand{\ram}{\operatorname{ram}}
\def\ra{\rangle}
\def\la{\langle}

\usepackage{soul,xspace}
\newcommand{\sch}[2]%
{\st{#1}{\color{DarkOrchid} \sf {#2}}}
\renewcommand{\sch}[2]{#2}
\newcommand\logpair{log variety\xspace}
\newcommand\logpairs{log varieties\xspace}
\newcommand\alogsmooth{an snc\xspace}
\newcommand\logsmooth{snc\xspace}
\newcommand\Logsmooth{Snc\xspace}
\newcommand\logres{log resolution\xspace}
\hypersetup{
    colorlinks=true,
    linkcolor={blue!50!black},
    citecolor={green!50!black},
    urlcolor={red!80!black}
}
\numberwithin{equation}{theorem}

\begin{document}

\title{The minimal model \sch{theory}{program} for
  \sch{}{b}-log canonical divisors and applications}
\author{Daniel Chan}
\address{Daniel Chan, Department of Pure Mathematics,
School of Mathematics and Statistics
UNSW Australia}
\email{danielc@unsw.edu.au}

\author{Kenneth Chan}
\address{Kenneth Chan, Department of Mathematics, University of Washington, Seattle, WA, USA}
\email{kenhchan@math.washington.edu}

\author{Louis de Thanhoffer de V\"{o}lcsey}
\address{Louis de Thanhoffer de V\"{o}lcsey, Department of Mathematics,
  University of Toronto, Toronto, ON, Canada}
\email{louis.dethanhofferdevolcsey@utoronto.ca}

\author{Colin Ingalls}
\address{Colin Ingalls, Department of Mathematics and Statistics,
University of New Brunswick,
Fredericton, NB, Canada}
\email{cingalls@unb.ca}%

\author{Kelly Jabbusch}
\address{Kelly Jabbusch, Department of Mathematics and Statistics, University of Michigan - Dearborn, Dearborn, MI, USA}
\email{jabbusch@umich.edu}

\author{S\'andor J Kov\'acs}
\address{S\'andor J Kov\'acs, Department of Mathematics, University of Washington, Seattle, WA, USA}
\email{skovacs@uw.edu}

\author{Rajesh Kulkarni}\address{Rajesh Kulkarni, Department of Mathematics,
Michigan State University,
East Lansing, MI, USA}
\email{kulkarni@math.msu.edu}

\author{Boris Lerner}\address{Boris Lerner, Department of Pure Mathematics,
School of Mathematics and Statistics
UNSW Australia}
\email{blerner@gmail.com}

\author{Basil Nanayakkara}\address{Basil Nanayakkara, Brock University, St.~Catharines, ON, Canada}
\email{bnanayakkara@brocku.ca}

\author{Shinnosuke Okawa}
\address{Shinnosuke Okawa, Department of Mathematics, Graduate School of Science, Osaka University, Osaka, Japan}
\email{okawa@math.sci.osaka-u.ac.jp}

\author{Michel Van den Bergh}\address{Michel Van den Bergh, Universiteit Hasselt,
Diepenbeek, Belgium}
  \email{michel.vandenbergh@uhasselt.be}

\begin{abstract}
  We discuss the minimal model \sch{theory}{program} for \emph{b-\logpairs}, which is
  a pair of a variety and a b-divisor, as a natural generalization of the minimal
  model \sch{theory}{program} for \sch{usual}{ordinary} \logpairs\sch{ where we have
    a boundary divisor on every model}.  \sch{We derive standard theorems from the
    usual log minimal model theory.  In particular, if log MMP terminates, then so
    does the b-log MMP.}{We show that the main theorems of the log MMP work in the
    setting of the b-log MMP.  If we assume that the log MMP terminates, then so does
    the b-log MMP.} \sch{Also,}{Furthermore, the} b-log MMP includes both \sch{}{the}
  log MMP and \sch{}{the} equivariant MMP as special cases.  There are various
  interesting b-\logpairs arising from different objects, including the \emph{Brauer
    pairs}, or ``non-commutative algebraic varieties which are finite over their
  centres''.  The case of toric Brauer pairs \sch{will be}{is} discussed in further
  detail.
\end{abstract}
\thanks{\emph{Mathematics Subject Classification (2010):} Primary 14E15, Secondary 16H10.}
\maketitle
\tableofcontents

\section{Introduction} Let $\bk$ be an algebraically closed field of characteristic
zero.  Let $K$ be a field, finitely generated over $\bk$.  A b-divisor $\bfD$
associates a $\bQ$-divisor $\bfD_X$ to every normal model $X$ of $K,$ compatibly with
pushforward.  We assume throughout that the coefficients of our b-divisors are
rational numbers in the interval $[0,1)$.  The main result of this paper is that
replacing the canonical divisor $K_X$ with $K_X+\bfD_X$ everywhere, provides a
generalization of the minimal model program, namely the b-log MMP.  The b-log MMP
includes the G-equivariant MMP and the log MMP as special cases, by using appropriate
b-divisors, as explained in Examples~\ref{eg:equivariant_MMP_as_b-MMP}
and~\ref{eg:log_MMP_as_b-MMP}.

We show that the main theorems of the log MMP work in the setting of the b-log MMP.
If we assume that the log MMP terminates, then so does the b-log MMP.  The
contractions and flips of the b-log MMP are simply log MMP contractions and flips for the \logpair
$(X,\bfD_X)$ and so many of the results for the b-log MMP are direct consequences of
those for the log MMP.

The b-log MMP differs from the log MMP in terms of what types of singularities are
permitted.  By using the b-divisor in the definition of discrepancy we obtain the
following formula for a birational proper morphism $f:Y \to X$
\[K_Y + \bfD_{ Y } = f ^{ * } \lb K_X + \bfD_X \rb + \sum_E b' ( E; X, \bfD ) E\]
where the sum is over $f$\sch{}{-}exceptional divisors.  Thus we obtain a modification of the usual discrepancy.  Let $d_E$ be the coefficient
of $E$ in the b-divisor $\bfD$ and let $r _{ E } = 1/(1 - d _{ E }).$ We also
introduce another modification of the discrepancy $b(E;X,\bfD) = r_E b'(E;X,\bfD)$
which is more natural from several points \sch{}{of view} as seen in
Corollary~\ref{cr:b-lt_b-lc_vs_a-lt_a-lc}, Corollary~\ref{cr:b-dlt_and_resolution},
Remark~\ref{rmk:bvsbprime} and Example~\ref{eg:equivariant_MMP_as_b-MMP}. Using this
definition of discrepancy we obtain notions of b-terminal, b-canonical, b-log
terminal and b-log canonical.

\sch{In order to begin the contractions of}{Running} the MMP \sch{}{usually starts
  with} \sch{one would first resolve}{resolving} singularities.  In our case, there is
no appropriate notion of smoothness so we must begin by resolving singularities to a
b-terminal model.  We show that any b-\logpair admits a b-terminal resolution of
singularities in Theorem~\ref{th:existence_of_Brauer_terminal_resolution} and
Corollary~\ref{secondproof}.  \sch{We in fact}{In fact, we} provide two proofs of
this result.  The first one is shorter and relies on~\cite{MR2601039}, and the second
proof is longer but is more constructive and uses toroidal geometry.

Once we have an appropriate \sch{}{partial} resolution, we can \sch{begin}{start
  running} the log MMP\sch{ contractions}{}.  The negativity lemma allows us to
conclude that contractions and flips preserve the type of singularities as shown in
Corollary~\ref{MMP_preserves_bterminal_etc}.  If the log MMP terminates then so does
the b-log MMP.  This establishes the main results of the b-log MMP.  Next, we discuss the
history and motivation of our application of b-log MMP to noncommutative algebraic
geometry.

It was noted by M.~Artin that given a maximal order $\Lambda$ over a variety $X$, a tensor power of the dualizing sheaf $\omega_\Lambda^{\otimes n}$ of $\Lambda$
could be realized as the pull back of a divisor $n(K_X+\Delta)$ on $X$ in codimension
one.  This suggested that one can use a $\Q$-divisor on $X$ for what would naturally
be considered the canonical divisor of $\Lambda$.  This idea was used by Chan and
Kulkarni in \cite{ChanKulkarni} to classify del Pezzo orders.  In
\cite{MR2180454}, Chan and Ingalls applied this idea and the log minimal model
program for surfaces to birationally classify orders over surfaces.  This is also
treated in \cite{ArtinChanLieblichdeJong}.  Since then, there remained the issue of
extending the results to higher dimension.  In \cite{Nanayakkara}, Nanayakkara,
showed that Brauer pairs $(X,\alpha)$ with $\alpha \in \Br X$ of order 2,
have b-terminal resolutions in all dimensions, allowing one to \sch{begin}{start} the
minimal model program for Brauer pairs by applying log MMP contractions and flips for the pair
$(X,\Delta)$.  However it was not clear if the steps of the MMP would preserve the notion of Brauer terminal, or if terminal resolutions existed in other cases.  In 2014, a
meeting was held at the American Institute of Mathematics, in order to solve this
problem.  This paper is a joint work of all the participants at that meeting.

We briefly describe the contents of this paper.  The second section begins by defining b-divisors and giving the examples of interest to us.  In particular, we discuss the proper transform b-divisor which recovers the log MMP and the ramification b-divisor of a Galois cover, which recovers the equivariant MMP.  We also describe other ramification b-divisors such as those arising from ramification of Brauer classes, which was the original motivation for this work.

In Section~\ref{sec:b-discrepancy}, we define b and b'-discrepancy and the associated
singularities, and relate these to the usual discrepancy.  In particular we note that
a b-\logpair $(X,\bfD)$ is b-log terminal or b-log canonical if and only if the
\logpair $(X,\bfD_X)$ is log terminal or log canonical\sch{.  Also,}{ and} a
b-\logpair is $(X,\bfD)$ is b-terminal or b-canonical if and only if it is
b'-terminal or b'-canonical.  We also establish the existence of b-terminal
resolutions of b-\logpairs in
Theorem~\ref{th:existence_of_Brauer_terminal_resolution}.  Next, in
Section~\ref{sec:blogMMP} we establish the main results of the b-log MMP.  Much of
this follows from the usual results of the log MMP.  The minimal models will be b-terminal, b-canonical, b-lt or b-lc if one first resolves to
that type of singularity before \sch{beginning contractions}{running the b-log MMP}.

In Section~\ref{sec:toroidal_blog_pairs}, we begin by discussing b-discrepancies for divisors over \logsmooth pairs.  Then we consider the case of toric b-\logpairs and their b-discrepancy in some detail.  We give a constructive proof of the existence of b-terminal resolutions in the toric case
in Proposition~\ref{prop:exists_bterm}.  We complete this section by using the toric results combined with toroidal geometry to provide another proof of the existence of b-terminal resolutions in Corollary~\ref{secondproof}.

In Section~\ref{toric_brauer_classes}, we return to our original motivation for b-divisors coming from ramification of Brauer classes.  We restrict to the case of toric Brauer classes.  Given a non-degenerate
toric Brauer class $\alpha$ with toric variety $X$,
we show that the b-\logpair $(X,\bfD_\alpha)$ is b-terminal, etc.~ if and only if $X$ is terminal, etc.~in Proposition~\ref{full_rank_toric}.  We
characterize the singularities of the b-\logpair $(\mathbb{A}^3,\bfD_\alpha)$ for a toric Brauer class $\alpha$.

We give an application of the b-log MMP.
\begin{corollary}
  Let $K$ be a field, finitely generated over $\bk$.  Let $\Sigma$ be a central
  simple $K$ algebra with Brauer class $\alpha \in \Br K$ and ramification b-divisor
  $\bfD_\alpha$.  Suppose there is a model $X$ such that  $K_X+\bfD_{\alpha,X}$ is big.  Then the group of outer automorphisms of $\Sigma$ is finite.
\end{corollary}
\begin{proof}
  By Theorem~\ref{MMPworks} and Remark~\ref{canonicalModelExists} there is a unique
  canonical model $X$ with $K_X+\bfD_{\alpha,X}$ nef and semi-ample.  We have an
  exact sequence
  \[1 \to \Sigma^*/K^* \to \Aut \Sigma \to \Out \Sigma \to 1.\]
  Note further, that the Skolem-Noether theorem shows that we have an injective map
  $\Out \Sigma \to \Bir (X,\bfD_\alpha)$.  By Iitaka's Theorem~\cite[Theorem
  11.12]{Iitaka1982} (see also~\cite[Theorem 1.2]{FG2014}
  and~\cite[\sch{Theorem}{Proposition} 6.5]{KP2017}) we have that
  $\Bir (X,\bfD_\alpha)$ is finite.\end{proof}

We also note that the ideas in this paper are used in~\cite{GrieveIngalls}, where two
related results are established.
\begin{theorem}~\cite[Theorem 1.3]{GrieveIngalls} Let $K$ be a finitely generated
  field with a b-divisor $\bfD$.  If $(X,\bfD)$ and $(Y,\bfD)$ have b-canonical singularities and $\ell(K_X+\bfD_X)$ and $\ell(K_Y +\bfD_Y)$ are both Cartier then
  \[
  \bigoplus_{n\geq 0} H^0(X,n\ell(K_X+\bfD_X))  =  \bigoplus_{n\geq 0} H^0(Y,n\ell(K_Y+\bfD_Y))
  \]
are naturally isomorphic rings.
\end{theorem}  This leads to a birationally invariant notion of Kodaira dimension for b-divisors.
\sch{Let $G$ be a finite group.}{} In addition, \sch{}{for a finite group $G$,} they
show the existence of $G$-equivariant b-terminal resolutions of b-log pairs~\cite[Theorem 4.15]{GrieveIngalls} using Theorem~\ref{th:existence_of_Brauer_terminal_resolution}
of this paper.

\begin{remark} Let $K$ be a field, finitely generated over $\bk$.  Let $\Sigma$ be a
  central simple $K$-algebra with Brauer class $\alpha$.  Let $\Lambda$ be a maximal
  order $\Sigma$ with ramification data $(X,\bfD_{\alpha,X})$ as in
  Example~\ref{eg:b-divisor_of_brauer_pair}.  We may run the minimal model program
  for $\Lambda$ in the following way.  We first resolve singularities of
  $(X,\bfD_\alpha)$ to a b-terminal model by using
  Theorem~\ref{th:existence_of_Brauer_terminal_resolution} \sch{providing}{obtaining}
  a birational morphism $f:Y \to X$.  We choose a maximal order $\Lambda_Y$
  containing $f^*\Lambda$.  Next, we \sch{carry out the contractions of}{run} the
  b-log MMP.  For a birational contraction or a flip $g:Y \dashrightarrow Y'$, we take
  reflexive hull $\Lambda_{Y'} = (g_* \Lambda_Y)^{\vee \vee}$ which will be a maximal
  order by~\cite[Theorem 1.5]{MR0117252}.  If the log MMP terminates in a birational
  model (not a Mori fibre space) then so does the b-log MMP and we will obtain a
  maximal order $\Lambda_Z$ on a b-terminal minimal model $(Z,\bfD_{\alpha})$.  The
  pair $(Z,\bfD_{\alpha,Z})$ is canonically determined by $\Sigma$ up to log flops.
  Note further that in dimension two by~\cite[Theorem 1.2]{MR2180454}, the order
  $\Lambda_Z$ is unique up to Morita equivalence.  This result relies heavily on the
  possible algebraic structure of the order in dimension two and we do not have a
  similar result for higher dimensions.  So we ask the following question.
  \end{remark}
  \begin{question} To what extent is the maximal order on a minimal model uniquely
    determined?
    \end{question}
  \begin{question}
    How do Mori fibre spaces for the b-\logpair $(X,\bfD_\alpha)$
    interact with a maximal order $\Lambda$ on $X$?  For instance, is there a semi-orthogonal decomposition of the derived category?
    \end{question}
We work over an algebraically closed field $\bk$
throughout the paper. The characteristic of
$
 \bk
$
will be assumed to be $0$ unless otherwise stated.
For a scheme $X$, the set of points of codimension $c$
will be denoted by
$
 X ^{ ( c ) }
$.
A \emph{variety} is an integral scheme which is separated and of finite type
over
$
 \bk
$.

\begin{acknowledgements}
  This paper is an outcome of the workshop `Mori program for Brauer pairs in
  dimension three' held in 2014 at the American Institute of Mathematics.  The
  authors are indebted to all the support from the institute and their hospitality.
  \sch{D.C.}{DC} was partially supported by an ARC discovery project grant
  DP130100513.  \sch{C.I.}{CI} was partially supported by an NSERC Discovery grant.
  \sch{}{SJK was supported in part by NSF Grant DMS-1565352 and the Craig McKibben
    and Sarah Merner Endowed Professorship in Mathematics at the University of
    Washington.}
RK was partially supported by the National Science Foundation award DMS-1305377.
  \sch{S.O.}{SO} was partially supported by Grant-in-Aid for
  Scientific Research (No.~60646909, 16H05994,
16K13746,
16H02141,
16K13743,
16K13755,
16H06337)
 from JSPS and the Inamori Foundation.
\end{acknowledgements}

%
%

\section{Minimal model theory with boundary b-divisors}

\subsection{Recap on b-divisors}

We recall the notion of \emph{b-divisors} after \cite[Section 2.3.2]{MR2359340} (`b'
stands for `birational').  We will change notation slightly, by not fixing a
particular model.  For standard terminologies related to singularities in the Minimal
Model \sch{Theory}{Program}, readers \sch{can}{may} refer to \cite[Section
2.3]{MR1658959} or \cite{SingBook}.

Let $K$ be a field, finitely generated over our base field $\bk$ and let
$\eta =\Spec K$.  A \emph{model} of $K$ is an irreducible variety $X$ over $\bk$ with a
fixed map $\eta \to X$ over $\bk$, \sch{which is an isomorphism at}{mapping $\eta$
  isomorphically to} the generic point of $X$.

The category of schemes over $\bk$ and under
$
 \eta
$
will be denoted by
\sch{%
$
{} _{ K } / \Sch _{ / \bk }
$.%
}
{%
$
{} _{ K/ } \Sch _{ / \bk }
$.%
}
We take ${} _{ K / } \cM _{ / \bk} $ to be the full subcategory of objects $X$ which are normal and proper models of $K$,
where the maps are given by birational morphisms that commute with the fixed map from $\eta$.
An object of ${} _{ K / } \cM _{ / \bk}$ will be called a \emph{(proper) model of $K$}.

\begin{definition}
  Let $E$ be a prime divisor in some normal model of $K$.  The divisor $E$ gives us a
  discrete valuation $\nu$ on $K$ such that $\trdeg \kappa(\nu) = \trdeg(K)-1$.
  Recall that a {\it place} is an equivalence class of valuations with equal
  valuation rings.  We will call such valuations and places {\it geometric}.  Let $R$
  be the discrete valuation ring of $\nu$ and let $\xi$ be the closed point in
  $\Spec R$.  Let $X$ be a normal proper model of $K$.  We have maps
  $\Spec R \leftarrow \eta \to X$.  Since $X$ is proper, we obtain a unique extension
  $\Spec R \to X$.  The closure of the image of the closed point $\xi \in \Spec R$ in
  $X$ will be denoted by $ C_X E =\overline{\lc \xi \rc}$ and called the \emph{centre
    of $E$ on $X$}.  There exists a normal model $Y$ of $K$ with a birational
  morphism $f:Y \to X$, where the centre of the valuation $\nu$ is an irreducible
  divisor $E$.  Since $Y$ is normal, we have that the local ring $\cO_{Y,E} =R.$
  \sch{The}{and the} closed subset $ f ( E ) \subset X $ is \sch{also}{} $ C_X E $.
A divisor $E$ in some model is \emph{exceptional} over a model $X$ if
$
 C _{ X } E
$
has codimension greater than 1 in $X$.
\end{definition}

The group of Weil divisors on a normal variety $X$ will be denoted by
$
 \Div X.
$
One can define the pushforward of a Weil divisor under a proper morphism of normal
varieties 
(\cite[Section 1.4]{MR1644323})\sch{.
T}{, t}hus we obtain a functor
\begin{equation}
 \Div \colon _{ K / }\cM _{ / \bk } \to \Ab; \quad X \mapsto \Div X,
 \quad \lb f \colon Y \to X \rb \mapsto \lb f _{ * } \colon \Div Y \to \Div X \rb
\end{equation}
to the category of abelian groups
$
 \Ab
$.
By restricting to effective divisors, we also obtain the functor
$
 \Div _{ \ge 0 }
$
in the obvious way.


\begin{definition}[{$=$\cite[Definition 2.3.8]{MR2359340}}]\label{df:b-divisor}
An element
$
 \bfD
$
of the limit object
\begin{equation}\label{eq:group_of_b-divisors}
 \bDiv ( K ) := \varprojlim _{ X \in {} _{ K / } \cM _{ / \bk} } \Div(X) \in \Ab
\end{equation}
will be called a(n integral) \emph{b-divisor} on $K$.
Similarly, an element of the subset
\begin{equation}
 \bDiv _{ \ge 0 } ( K ) := \varprojlim _{ X \in {} _{ K / } \cM _{ / \bk} } \Div_{ \ge 0 }(X)
\end{equation}
will be called an effective (integral) b-divisor on $K$.

  A b-divisor on $X$ \sch{is equally}{may equivalently be} described as a formal
  integral sum
\begin{equation}
 \bfD = \sum _{ \Gamma } d _{ \Gamma } \Gamma,
\end{equation}
where $ \Gamma $ runs through all the geometric places of $K$, such that for each
normal model $X$ there are only finitely many $ \Gamma $ whose centre on $X$ is
divisorial and $ d _{ \Gamma } \neq 0 $.  A b-divisor \sch{}{$\bfD$} associates a
divisor \sch{on}{to} every normal model $X$ of $K$, \sch{so we have}{which is called}
the \emph{trace} of $\bfD$ on $X$ defined by the natural projection map
$\tr_X \colon \bDiv(K) \to \Div(X).$ Write $X^{(1)}$ for the set of irreducible divisors in $X$, or equivalently the set of codimension one points. 
We write
\begin{equation}
 \bfD _{ X } = \tr _{ X } \bfD = \sum _{ \Gamma  \in X^{(1)} } d _{ \Gamma } \Gamma
\end{equation}
(see \cite[Notation and Conventions 2.3.10]{MR2359340}).
Note that this is a finite sum for any particular model, and given a birational morphism $f:Y \to X$ we have $f_*\bfD_Y = \bfD_X$.  The b-divisor
$
 \bfD
$
is effective if and only if all the coefficients
$
 d _{ \Gamma }
$
 are non-negative.  Note that we can also interpret a b-Divisor $\bfD$ as function $d$ which associates a number $d_\nu$ to every geometric place $\nu$ of the field $K$ such that for any model $X$, the support of $d$ restricted to the divisors of $X$ is finite.  We will refer to the value of this function $d_\nu,$ or $d_E,$  on a geometric place $\nu$, or a divisor $E,$ as the {\it coefficient} of $\bfD$ along $E$.  In addition,  if a b-divisor $\bfD$ is defined for all models $Y$ over a fixed model $X$, then it extends naturally to all models.  Indeed, given any model $Z$, we can find a common model $Y$ with birational morphism
 $Y \to X$ and $f:Y \to Z$ and so the trace on $Z$ is given by
 \begin{equation}\label{pushforward:eqn} \bfD_Z  = f_*(\bfD_Y). \end{equation}
\end{definition}



We will freely extend the coefficients of b-divisors to $\bQ$.
All the notions defined so far are naturally extended to b-$\bQ$-divisors.
We will work primarily with b-divisors with rational coefficients so we will refer to them simply as b-divisors and we will write $\bDiv K$ for the set of b-divisors with rational coefficients.
Our goal is to develop the minimal model program for {\it b-\logpairs} $(X,\bfD)$ where $X$ is a normal proper variety and $\bfD$ is a b-divisor in $\bDiv(\bk(X))$ with coefficients in $[0,1) \cap \bQ$.
If all coefficients of the b-divisor $\bfD$ are
contained in the interval
$
 [0, 1) \cap \bQ
$, or equivalently, $\lfloor \bfD \rfloor =0$,  we will call the  b-divisor {\it fractional.}

First we will consider some motivating examples of b-divisors that occur naturally.
Recall that a divisor $D$ is $\bQ$-{\it Cartier} if there is a non-zero rational number $a$ such that $aD$ is a Cartier divisor. 

\begin{example}\label{eg:Cartier_closure}
Given a
$
 \bQ
$-Cartier divisor
$
 D
$
on $X$, its \emph{Cartier closure}
$
 \Dbar
$
is the b-divisor whose trace on a model $Y$ over $X$ given by
$
 f \colon Y \to X
$
is
$
 \Dbar_Y = f ^{ * } D
 $.  We extend $\Dbar$ to all models by pushforward as described above Equation~\ref{pushforward:eqn}.
\end{example}

\begin{example}\label{eg:principal_b-divisor}
Take a non-zero rational function
$
 \varphi \in K^{\times}.
$
  We associate a b-divisor $\bdiv(\varphi)$ in $\bDiv(K)$ 
whose trace on the model
$
 X
$
is defined by
\begin{equation}
 \bdiv ( \varphi )_{ X } := \divi_X  ( \varphi ).
\end{equation}
This will be called the \emph{principal b-divisor} associated to
$
 \varphi
$.
The equality
\begin{equation}
 \myov{ \divi _{ X } ( \varphi ) } = \bdiv ( \varphi ),
\end{equation}
where the left hand side is the Cartier closure of the Cartier divisor
$
 \divi _{ X } ( \varphi )
$,
is easily seen.
\end{example}

Among others, canonical b-divisors play quite an important role in this paper.
\begin{example}\label{eg:canonical_b-divisor}
Fixing a rational differential
$
 \omega \in \bigwedge ^{ \trdeg K } \Omega _{ K/k }
$
defines a \emph{canonical b-divisor}
$
 \bfK = \bdiv _{ X } \lb \omega \rb
$
on
$
 X
$.
On each model
$
 X
$,
the trace will be defined as
$
 \divi _{ X } \lb \omega \rb
$
associated to the rational global section
$
 \omega
$
of the canonical sheaf
$
 \cO _{ X } \lb K _{ X} \rb
$.
\end{example}

\begin{remark}
In the example above and the lemma below, a canonical divisor on $X$ means a specific choice of a Weil divisor on $X$
(not its linear equivalence class in the Weil divisor class group).
\end{remark}

\begin{remark}
  Given a canonical b-divisor
  $
\bfK \in \bDiv(K)
$,
for any model
$
X
$,
we will write
$
 K _{ X } = \bfK _{ X }
$.
\end{remark}

\begin{lemma}
A canonical b-divisor is uniquely determined by its trace on
any fixed model.
\end{lemma}

\begin{proof} Let $X$ be a fixed model, and
fix the trace $K_X$ of a canonical b-divisor.  Let
$
 f \colon Y \to X
$
be a model over
$
 X
$,
and
$
 K _{ Y }, K' _{ Y }
$
be two canonical divisors on $Y$
such that
$
 f _{ * } K _{ Y } = K _{ X } = f _{ * } K' _{ Y }
$.
Then
$
 K _{ Y } - K' _{ Y } = \divi _{ Y } ( \varphi )
$
for some
$
 \varphi \in \bk ( Y )  = \bk ( X )
$
and the support of $K_Y - K_{Y'}$ is contained in the exceptional locus of
$
 f
$.
Since
$
 \divi _{ Y } ( \varphi ) = f ^{ * } \divi _{ X } ( \varphi )
$
and
$
 \divi _{ X } ( \varphi ) = f _{ * }  \divi _{ Y } ( \varphi ) = 0
$
by the assumption, we see
$
  \divi _{ Y } ( \varphi ) = 0
$.
\end{proof}


\begin{example}
Consider a Weil divisor
$
 D
$
on
$
 X
 $.
The \emph{proper transform b-divisor}
$
\Dhat
$
is the b-divisor whose trace on a model
$
 f \colon Y \to X
$
is defined by
$
 \Dhat_{ Y } = \lb f ^{ - 1 } \rb _{ * } D
 $, and naturally extended to all models via push-forward.  Note that the coefficient of $\Dhat$ on any exceptional divisor over $X$ is zero.  In fact, proper transform b-divisors are characterized by the support of the formal sum $\bfD = \sum d_\Gamma \Gamma$ over all geometric places $\Gamma$ being finite. 
\end{example}

\begin{example}\label{eg:b-divisor_of_nonabelian_H1}
Let
$
 G
$
be a finite group. Recall that an element
$
 [ L ] \in H ^{ 1 }\lb K, G \rb
$
is represented by an isomorphism class of a Galois extension
$
 L / K
$
with a homomorphism  $\Gal(L/K) \to G$ which we can assume to be injective.
Let $X$ be a model of $K$ and let 
 $\pi \colon \Xtilde \to X
$
be the normalization of
$
 X
$
in the field
$
 L
$, so that the field homomorphism
$
 \pi ^{ * } \colon \bk ( X ) \to \bk ( \Xtilde ) \simto L
$
is canonically identified with the extension
$
 L / K
$.

Since
$
 \pi \colon \Xtilde \to X
$
is again a Galois extension with Galois group a subgroup of
$
 G
$, the Riemann-Hurwitz Theorem tells us that 
there exists an effective
$
 \bQ
$-divisor
$
 \bfD _{ X }
$
on
$
 X
$
such that
\begin{equation}
 K _{ \Xtilde } = \pi _{ X } ^{ * } \lb K _{ X } + \bfD _{ X } \rb
\end{equation}
as $\bQ$-divisors on $\Xtilde$. One can easily verify that the divisors
$
 \bfD _{ X }
$
give rise to a fractional b-divisor
$
 \bfD \in \bDiv _{ \ge 0 } \lb K \rb
$,
 which will be called the \emph{ramification b-divisor}. 

 On the other hand, let $n = \trdeg K$ and fix $\omega \in \Omega^n_{K/\bk}$ and consider
   $\omega \otimes 1 \in \Omega^n_{K/\bk} \otimes_K L \simeq \Omega^n_{L/\bk}.$
   We associate canonical b-divisors
   $
 \bfK = \bdiv  \lb \omega \rb
$ and $\bfKtilde = \bdiv \lb \omega\otimes 1 \rb$
in $\bDiv (K )
$
and
$
\bDiv ( L )
$ respectively.
Then we have the equality of b-divisors
$
 \pi ^{ * } \lb \bfK + \bfD \rb = \bfKtilde
$.
As we will see later in \pref{eg:equivariant_MMP_as_b-MMP}, the MMP for b-\logpairs applied to the pair
$
 ( X, \bfD )
$
is equivalent to the $G$-equivariant MMP for
$
 \Xtilde
$.

\end{example}

\begin{example}\label{eg:b-divisor_of_brauer_pair}
This example is the original motivation for the authors to establish the Minimal Model Theory
for b-\logpairs.
Let
$
 K
$
be a field, finitely generated over $\bk$ and
let $
 \alpha \in H ^{ 2 } \lb K, \bG _{ m } \rb = \Br K
$
be a Brauer class.   A \emph{Brauer pair}, 
$
 ( X, \alpha )
$
is a pair of a normal proper model $X$ of $K$ and an $\alpha \in \Br K$.
 Then we can define the effective divisor
\begin{equation}
 \bfD _{ \alpha,X} = \sum _{ D \in X ^{ ( 1 ) } } \lb 1 - \frac { 1 } { r _{ D } }\rb D ,
\end{equation}
where
$
 r _{ D } \in \bZ _{ \ge 1 }
$
is the \emph{ramification index} of the Brauer class
$
 \alpha \in \bk ( X )
$
along the prime divisor
$
 D
$,
 which is defined via the Artin-Mumford map \cite{ArtinMumford}.  Given $\alpha \in \Br (K)$ we have
 \[ H^2(\bk(X),\mathbb{G}_m) \stackrel{\ram}{\to} \bigoplus_{ D \in X^{(1)}} H^1(\bk(D),\Q/\Z)\]  and we define $r_D$ above to be the order of $\ram_D(\alpha)$.  The divisors $\bfD_{\alpha,X}$ give a fractional b-divisor.
 We note that the divisor $\bfK_X +\bfD_{\alpha,X}$ can be viewed as the canonical divisor of a maximal order in a central simple $K$ algebra representing $\alpha,$ as noted in \cite{MR2180454}, or can be interpreted as the canonical divisor of the associated root stack \cite[Appendix B]{rootstack}.

 We also note that,  in Example~\ref{eg:b-divisor_of_nonabelian_H1}, if we have a cyclic Galois cover, we can treat it analogously to a Brauer class, if we use the map
\begin{align}
 H^1(\bk(X), \mu) \stackrel{\ram}{\to} \bigoplus_{ D \in X^{(1)}} H^0(\bk(D),\Q/\Z)
\end{align}
to define the coefficients of the ramification b-divisor.
\end{example}

\begin{example}
  The above example can be generalized to the setting of Rost modules.  This includes
  algebraic $K$\sch{}{-}theory, Chow cohomology, motivic cohomology, and more.  In
  \cite{Rost}, the notion of Rost (cycle) modules is defined.  Given a Rost module
  $M$ and a normal scheme $X$, we obtain maps
  $\partial_D : M ( \bk ( X ) ) \to M(k(D))$ for all irreducible divisors $D$ in $X.$
  Given an element $\alpha \in M ( \bk ( X ) )$ only finitely many
  $\partial_D(\alpha)$ are non-zero as in Definition 2.1 of \cite{Rost}.  So given
  such an $\alpha$, if the $\partial_D(\alpha)$ has finite order $r_D$ for all $D$,
  (for example if $\alpha$ has finite order), we can define a ramification b-divisor
  by
 \begin{equation}
 \bfD _{ \alpha,X } = \sum _{ D \in X ^{ ( 1 ) } } \lb 1 - \frac { 1 } { r _{ D } } \rb D.
\end{equation} 
This also includes the case of abelian Galois covers from Example~\ref{eg:b-divisor_of_nonabelian_H1}.
\end{example}

%
%

\subsection{b-discrepancy}\label{sec:b-discrepancy}

In this section we introduce the discrepancy for b-divisors.  
First we will recall some facts about the usual notion of discrepancy before we introduce our modification for b-divisors.
Recall the following definition.
\begin{definition}
  Let $(X,D)$ be a $\Q$-Gorenstein \logpair and let $f:Y \to X$ be a birational morphism.
  The discrepancy of divisors $E$ in $Y$ that are exceptional over $X$ for the \logpair $(X,D)$ are defined by the equation 
  \[ K_Y +f^{-1}_*D = f^*(K_X+D) +\sum_{E} a(E;X,D) E\]
  where the sum is taken over $f$\sch{}{-}exceptional divisors $E$, and $f^{-1}_*D$
  denotes the proper transform of $D$.  The discrepancy only depends on the divisor
  and not the choice of model $Y,$ as reflected in the notation.
\end{definition}

%

%





%

%
%

\begin{definition}\label{df:b-log_pair}
A \emph{b-\logpair} is a pair
$
 \lb X, \bfD \rb
$
of a normal variety $X$ and an effective b-$\bQ$-divisor $\bfD$ on $X$.  If
$
 K _{ X } + \bfD _{ X }
$
is $\bQ$-Cartier we say that the pair $(X,\bfD)$ is $\bQ$-Gorenstein.
The b-divisor
$
 \bfK + \bfD
$
will be called the \emph{log canonical b-divisor} of the pair
$
 ( X, \bfD )
$.
\end{definition}

In the rest of this paper, unless otherwise stated, we  assume
that all b-divisors are {\it fractional.}  Recall that this means all coefficients are in $[0,1) \cap \bQ$.  We will also tacitly assume all pairs are $\bQ$-Gorenstein, unless otherwise stated.

\begin{definition}\label{ramification:index:defn}
Let
$
 \lb X, \bfD \rb
$
be a fractional b-\logpair. For each divisor
$
 E
$
over
$
 X
$,
let
$
 d _{ E } \in [ 0, 1 ) \cap \bQ
$
be the coefficient of
$
 \bfD
$
along
$
 E
$.
The \emph{ramification index}
$
 r _{ E } \in [1, \infty ) \cap \bQ
$
of
$
 \bfD
$
along
$
 E
$
is defined by the equivalent equations:
\begin{align}\label{eq:ramification_index_of_b-divisor}
 r _{ E } = \frac{1}{1 - d _{ E }} \quad\quad d_E = 1 - \frac{1}{r_E}.
\end{align}

\begin{definition}\label{df:discrepancy_of_b-log_pair}
Let
$
 ( X, \bfD )
$
be a $\bQ$-Gorenstein b-\logpair and
$E$ an exceptional divisor over $X$.
Take a model
$
 f \colon Y \to X
$
such that the centre
$
 C_Y E \subset Y
$
is a divisor.
Then there exists \sch{}{a $b' ( E; X, \bfD )\in \bQ$ such that}
the following equality of
$
 \bQ
$-divisors
\begin{align}\label{eq:definition_of_b'-discrepancy}
\lb \bfK + \bfD \rb _{ Y }
=
f ^{ * } \lb \bfK + \bfD \rb _{ X } + b' ( E; X, \bfD ) E
\end{align}
\sch{}{holds}
on an open neighbourhood of the generic point of
$
 E \subset Y
$.
The rational number
$
 b' ( E; X, \bfD )
$
will be called the
\emph{b'-discrepancy} of the b-\logpair
$
 ( X, \bfD )
$
with respect to the divisor
$
 E
$
over
$
 X
$.
\end{definition}

\begin{example}\label{proper_transform_same_discrepancy}
Consider a usual \logpair
$
 ( X, D )
$
and the proper transform b-divisor
$
 \Dhat
$.
Then it follows from the definition that for any exceptional divisor
$
 E
$
over $X$,
\begin{equation}\label{eq:discrepancy_as_b-discrepancy}
 a ( E; X, D ) = b' ( E; X, \Dhat ) = b(E;X,\Dhat).
\end{equation}
In this sense, for exceptional divisors,
the usual discrepancy can be regarded as the b-discrepancy of a proper transform b-divisor.

Moreover, when
$
 D = 0
$,
 the equality \eqref{eq:discrepancy_as_b-discrepancy}
is valid for \emph{any} divisor over $X$;
recall that a geometric valuation of
$
 \bk ( X )
$
which admits a centre on $X$ is called exceptional if and only if
its centre on $X$ is not divisorial.
\end{example}

\begin{remark}\label{rm:comparison_of_discrepancies}
For a divisor
$
 E
$
over $X$ we have the equality
\begin{align}\label{eq:b-discrepancy_vs_usual_discrepancy}
 b' ( E; X, \bfD ) = a ( E; X, \bfD _{ X } ) + d _{ E },
\end{align}
where
$
 a ( E; X, \bfD _{ X } )
$
is the usual discrepancy of the \logpair
$
 ( X, \bfD _{ X } )
$
with respect to the divisor
$
 E
$.
In particular, if
$
 \bfD
$
is effective, we always have the inequality
\begin{align}
b' ( E; X, \bfD ) \ge a ( E; X, \bfD _{ X } ).
\end{align}
Equality holds precisely if
$\bfD$ is not supported on $E$.
\end{remark}

It is more natural to consider a slight modification of b'-discrepancy.  This modification is motivated by Corollary~\ref{cr:b-lt_b-lc_vs_a-lt_a-lc}, Corollary~\ref{cr:b-dlt_and_resolution}, Remark~\ref{rmk:bvsbprime} and Example~\ref{eg:equivariant_MMP_as_b-MMP}.
The \emph{b-discrepancy} of the b-\logpair
$
 ( X, \bfD )
$
with respect to the divisor
$
 E
$
over
$
 X
$
is defined by either of the equivalent equations
\begin{align}\label{definition_of_b-discrepancy}
  b ( E; X, \bfD ) = b' ( E; X, \bfD ) \cdot r _{ E }\\
\label{logdiscrepancy}  b(E;X,\bfD)+1 =r_E(a(E;X,\bfD)+1).
\end{align}
\end{definition}
Note that one can interpret the second equation above as saying that the b-log discrepancy is a positive multiple of the a-log discrepancy.

We say that the b-\logpair $( X, \bfD )$ is \emph{\logsmooth} if the associated pair
$( X, \bfD _{ X })$ is \sch{a}{} \logsmooth.  The following lemma will be frequently
used in this paper.
\begin{lemma}\label{lm:log_smooth_resolution}
For any b-\logpair
$( X, \bfD )$,
consider any \logres
$
 f \colon Y\to X
$
of the \logpair
$
 ( X, \bfD _{ X } )
$.
Then
$
 ( Y, \bfD_Y )
$
is \logsmooth.
\end{lemma}

\begin{proof}
Let
$
 f \colon Y \to X
$
be a \logres of the pair
$
 ( X, \bfD _{ X } )
$.
Then, by definition,
$
 \Exc { ( f ) }
 \cup
 ( f ^{ - 1 } ) _* \bfD _{ X }
$
is an \sch{SNC}{snc} divisor. Since
$
 \Supp \bfD _{ Y }
$
is a subset, it is \sch{SNC}{snc} as well.
\end{proof}

\begin{definition}\label{df:notions_of_singularities_for_b-log_pairs}
Let
$
 ( X, \bfD )
$
be a $\bQ$-Gorenstein b-\logpair. 
The \emph{minimal b-discrepancy} of the pair
$
 ( X, \bfD )
$
is defined by
\begin{align}
 \bdiscrep ( X, \bfD )
 : =
 \inf \lc b ( E; X, \bfD ) \mid E \ \mbox{is an exceptional divisor over } X \rc.
\end{align}
Note that the infimum is among all divisors over
$X$ which are exceptional.

We say
\begin{align}
( X, \bfD ) \ \mbox{is}
\begin{cases}
\mbox{b-terminal}\\
\mbox{b-canonical}\\
\mbox{b-log terminal (b-lt)}\\
\mbox{b-log canonical (b-lc)}\\
\end{cases}
\mbox{if} \ 
 \bdiscrep { ( X, \bfD ) }
\begin{cases}
 > 0\\
 \ge 0\\
 > - 1\\
 \ge - 1.\\ 
 \end{cases}
\end{align}

We also make corresponding definitions using $b'(E;X,\bfD)$ in place of $b(E;X,\bfD)$ and so will refer to b-\logpairs $(X,\bfD)$ as being $b'$-terminal, $b'$-canonical, $b'$-log terminal, or $b'$-log canonical.

Similarly we say
$
 ( X, \bfD )
$
is \emph{b-Kawamata log terminal (b-klt)} if it is b-lt and
fractional.
Finally we define the notion of \emph{b-dlt} pairs as follows;
a b-\logpair
$
 ( X, \bfD )
$
is \emph{b-divisorially log terminal (b-dlt)} if
there exists a \logres
$
 f \colon Y \to X
$
of
$
 ( X, \bfD )
$
such that
\begin{equation}\label{eq:b-dlt}
 b ( E; X, \bfD ) > - 1
\end{equation}
holds for any $f$-exceptional divisor $E$.
\end{definition}

\begin{lemma}\label{lm:a_and_b-discrepancies}
Let
$
 ( X, \bfD )
$
be a fractional $\bQ$-Gorenstein b-\logpair, and
$
 E
$
be a divisor over
$
 X
$.
Then
$
  a ( E; X, \bfD _{ X } ) > ( \mbox{resp. }\ge ) - 1
 \iff
  b ( E; X, \bfD ) > (  \mbox{resp. }\ge ) - 1
$.
\end{lemma}

\begin{proof}
  This follows immediately from Equation~\ref{logdiscrepancy} in the definition of b-discrepancy.
\end{proof}

\pref{lm:a_and_b-discrepancies} immediately implies the following corollaries.

\begin{corollary}\label{cr:b-lt_b-lc_vs_a-lt_a-lc}
Let
$
 ( X, \bfD )
$
be a fractional b-\logpair. Then
$
 ( X, \bfD )
$
is b-lt (resp. b-lc) if and only if the \logpair
$
 ( X, \bfD _{ X } )
$
is lt (resp. lc) in the usual sense.
\end{corollary}

\begin{corollary}\label{cr:b-dlt_and_resolution}
A b-\logpair
$
 ( X, \bfD )
$
is b-dlt if and only if \sch{}{$b ( E; X, \bfD ) > - 1$ holds for}
any exceptional divisor 
$
 E
$
over $X$ whose centre on $X$ is contained in the non-\sch{SNC}{snc} locus of
$
 ( X, \bfD _{ X } )
$\sch{satisfies {\eqref{eq:b-dlt}}}{}.
\end{corollary}

\begin{proof}
  The equivalence of corresponding conditions for a-discrepancy ($=$ equivalence of
  two different definitions of the notion of dlt pairs) is well known
  \cite[Proposition 2.\sch{40}{44}]{MR1658959}\sch{}{, \cite{MR1322695}}. On the
  other hand, one can immediately check that each of them is respectively equivalent
  to the b-counterpart because of \pref{lm:a_and_b-discrepancies}.
\end{proof}

\begin{remark}  We also note that b-\logpair $(X,\bfD)$ is b'-terminal (resp. b'-canonical) if and only if it is b-terminal (resp. b-canonical).  This follows immediately from definition~\ref{definition_of_b-discrepancy}.
\end{remark}

\begin{remark}\label{rmk:bvsbprime}
  If $(X,\bfD_X)$ is not log canonical then its discrepancy is equal to
  $-\infty$.  This observation allows us to see that it is also true that
$
 ( X, \bfD )
$
is b'-lc if and only if
$
 ( X, \bfD _{ X } )
$
is lc.  Similarly, if
$
 ( X, \bfD _{ X } )
$
is klt, then
$
 ( X, \bfD )
$
is b'-klt. On the other hand, as the following example shows, the converse does not hold. 

Let $ X $ be a cone over an elliptic curve $ E $ and let $ f \colon Y \to X $ \sch{be}{}
its minimal resolution.  \sch{The}{Note that the} exceptional divisor is isomorphic
to $ E $.  Let $ \bfD $ be a b-divisor on $ X $ such that $ \bfD _{ X } = 0 $ and the
coefficient of $\bfD$ along $E$ satisfies $ d _{ E } > 0 $.  Then one can check that
$ ( X, \bfD ) $ is b'-lt, though $ X $ is (strictly) lc.  Actually one can find a
Brauer class $ \alpha \in \Br ( \bfk ( X ) ) $ whose ramification along $E$
corresponds to an \'etale double cover of $E$, so that the associated b-divisor
$ \bfD_\alpha $ has $ d_E= \frac{1}{2} $.  One can similarly check that
\pref{cr:b-dlt_and_resolution} is not true for b'-discrepancy.
\end{remark}

%
%

\begin{example}\label{proper_transform_same_singularites}
Let
$
 ( X, D )
$
be a fractional \logpair
and consider the proper transform b-divisor
$
 \Dhat
$.
Then
\begin{align}
( X, D ) \ \mbox{is}
\begin{cases}
\mbox{terminal}\\
\mbox{canonical}\\
\mbox{klt}\\
\mbox{purely log terminal}\\
\mbox{dlt}\\
\mbox{log canonical}\\
\end{cases}
\iff
( X, \Dhat ) \ \mbox{is}
\begin{cases}
\mbox{b-terminal}\\
\mbox{b-canonical}\\
\mbox{b-klt}\\
\mbox{b-log terminal}\\
\mbox{b-dlt}\\
\mbox{b-log canonical}\\
\end{cases}
\end{align}
(see \cite[Definition 2.34]{MR1658959}).
\end{example}


%
%


In order to \sch{begin}{run} the b-log MMP with b-terminal singularities, it is
necessary to first resolve singularities to a b-terminal model.  The existence of
such a resolution is established in
Theorem~\ref{th:existence_of_Brauer_terminal_resolution}, and the proof of this
theorem is the goal of the rest of this section.

\begin{lemma}\label{lm:brauer_discrepancy_under_extraction_of_bad_divisor}
Let
$
 ( X, \bfD )
$
be a $\bQ$-Gorenstein b-\logpair and
$
 f \colon Y \to X
$
a model on which the trace
$
 \lb \bfK + \bfD \rb _{ Y }
$
is
$
 \bQ
$-Cartier.
Suppose
$
 b(E;  X, \bfD )\le 0
$
holds for any
$f$-exceptional prime divisor
$E$.
Then for any exceptional divisor
$F$
over
$Y$, we have the inequality
\begin{align}
b(F;  Y, \bfD ) \ge b(F;  X, \bfD ).
\end{align}
\end{lemma}

\begin{proof}
We show the claim for b'-discrepancies, since it is equivalent.
Let
$
 g \colon Z \to Y
$
be a model over $Y$ on which the centre of
$F$
is divisorial. Then around the generic point of
$F$ we have
\begin{align}
\begin{split}
K_Z+ \bfD _{ Z }
=
g^*(K_Y+  \bfD _{ Y })
+b'(F;  Y, \bfD  )F\\
=
g^* \lb f^*(K_X+  \bfD _{ X })+ \sum_E b'(E;  X, \bfD )E \rb
+b'(F;  Y, \bfD  )F\\
=
(f\circ g)^*(K_X+  \bfD _{ X })
+ \lb\lb\sum_E b'(E;  X, \bfD )m_F E\rb + b'(F;  Y, \bfD _ ) \rb F.
\end{split}
\end{align}
Since
$b'(E;  X, \bfD )\le 0$
and
$m_F E\ge 0$
hold for all
$E$,
we see
\[
b'(F; X, \bfD ) =\sum_E b'(E; X, \bfD )m_F E + b'(F; Y, \bfD  ) \le b'(F; Y,
\bfD  ).\qedhere
\] 
\sch{Thus we conclude the proof.}{}
\end{proof}

\begin{remark}\label{rm:finiteness_of_non_terminal_witnesses}
  Assume that the pair $( X, \bfD )$ is b-lt and fractional, so that the associated
  pair $ ( X, \bfD _{ X } ) $ is klt. Then by \cite[Proposition 2.36(2)]{MR1658959}
  and \pref{rm:comparison_of_discrepancies}, there are only finitely many exceptional
  divisors over $X$ with non-positive b-discrepancies.

\end{remark}

We will provide a second proof of the result below in Corollary~\ref{secondproof}.
The second proof uses toroidal geometry and is longer but \sch{}{it} is also more
explicit \sch{, and is}{and} more elementary in the sense \sch{}{that} it does not use
the result \cite[Exercise 5.41]{MR2675555} which depends on \cite{MR2601039}.

\begin{theorem}\label{th:existence_of_Brauer_terminal_resolution}
Let
$
 ( X, \bfD )
$
be a fractional b-\logpair such that
$X$ is a quasi-projective variety over
$\bk$.
Then there exists a projective birational morphism
$
 f \colon Y \to X
$
such that the b-\logpair
$
 ( Y, \bfD )
$
is b-terminal and $Y$ is $\bQ$-factorial.
\end{theorem}

\begin{proof}
By
\pref{lm:log_smooth_resolution},
we find a projective \logres
$X_1 \to X$
so that the pair
$
 ( X_1, \bfD _{  X _{ 1 } } )
$
is \logsmooth. Since $X_1$ is \logsmooth and $\bfD_{X_1}$ is fractional, the pair $
 ( X_1, \bfD _{  X _{ 1 } } )
$ is klt.   As noted in
\pref{rm:finiteness_of_non_terminal_witnesses},
there are only finitely many exceptional divisors
over
$X_1$
whose b-discrepancies are non-positive.
Let
$S$
be the set of such divisors. 
Then we can use \cite[Exercise 5.41]{MR2675555}
to obtain a birational projective morphism
\begin{equation}
 g \colon Y \to X _{ 1 }
\end{equation}
from a normal $\bQ$-factorial variety $Y$ such that the set of
$g$-exceptional divisors is exactly the set $S$.
By \pref{lm:brauer_discrepancy_under_extraction_of_bad_divisor},
we see that the b-\logpair
$
 ( Y, \bfD )
$
is b-terminal.
\end{proof}

Given a b-\logpair $(X,\bfD),$ we call the pair $(Y,\bfD),$ supplied by the above Theorem~\ref{th:existence_of_Brauer_terminal_resolution}, a {b-terminal resolution} of $X$.  Note that a b-log terminal resolution need not be \logsmooth, and it is not clear if any b-\logpair $(X,\bfD)$ admits a resolution which is simultaneously \logsmooth and b-terminal.

\begin{question}
  Let $(X,\bfD)$ be a fractional b-\logpair.  Is there always a projective birational morphism $Y \to X$ such that the b-\logpair $(Y,\bfD)$ is \logsmooth and b-terminal?
\end{question}

%
%

\section{The Minimal Model {Program} for b-\logpairs}\label{sec:blogMMP}

We define the notions for the  minimal model program
for b-\logpairs. For the corresponding notions
of (log) MMP, see
\cite[3.31]{MR1658959}.

\begin{definition}
Let
$( X, \bfD )$ be a fractional b-lc pair,
$X$ be quasi-projective and $\bQ$-factorial, equipped with
a projective morphism
$
 \pi \colon X \to U
$
to a normal quasi-projective variety $U$.

\begin{itemize}
\item
The pair
$
 ( X, \bfD )
$
is \emph{$\pi$-minimal} 
if
$
 K_X + \bfD _{ X }
$
is $\pi$-nef.  Note that the definition of a minimal model does not depend on the type of singularities of the pair.

\item
An \emph{extremal contraction} of
$( X, \bfD )$
over
$U$
is a morphism
$
 f \colon X \to Y
$
over $U$ which is an extremal contraction of the lc pair
$
 ( X,  \bfD _{ X } )
$.
We say $f$ is \emph{divisorial/small/a Mori fibre space} if $f$
is divisorial/small/a Mori fibre space in the usual sense.

\item A {\it flip} of the pair $(X,\bfD)$ over $U$ is a birational map
  $X \dashrightarrow X'$ over $U$ which is a flip of the pair $(X,\bfD_X)$ in the
  usual sense.  \sch{}{Note that this is consistent with $\bfD$ being a
    b-divisor. Since $X \dashrightarrow X'$ is an isomorphism in codimension $1$ on
    both $X$ and $X'$, it follows that $\bfD_{X'}$ is necessarily the proper
    transform of $\bfD_X$ on $X'$.}
  
\item
A \emph{minimal model program of
$( X, \bfD )$ over (or relative to) $U$}
is a sequence of birational maps over $U$
\begin{equation}\label{eq:b-mmp}
 X
 =
 X_0 \dasharrow X_1 \dasharrow \cdots \dasharrow X_n
\end{equation}
which is a minimal model program of the \emph{usual} lc pair
$
 ( X,  \bfD _{ X } )
$ over $U$
(see \pref{cr:b-lt_b-lc_vs_a-lt_a-lc}).
\end{itemize}
\end{definition}

\begin{remark}
If
$\varphi \colon X \dasharrow Y$
is either a divisorial contraction or a flip
of the b-\logpair
$
 ( X, \bfD )
$
over $U$, then clearly
$
 \bfD _{ Y }
 =
 \varphi_* \bfD _{ X }
$.
Therefore any subsequence
\[X_i \dasharrow X_{i+1} \dasharrow \cdots \dasharrow X_{j}
\]
of \eqref{eq:b-mmp} is a b-\sch{}{log }MMP for the pair $ ( X_i, \bfD ) $.

\end{remark}

\begin{lemma}
\label{lm:negativity_lemma_for_b-discrepancy}
Let
$\varphi \colon X \dasharrow Y$
be either a divisorial contraction or a flip of the b-\logpair
$( X, \bfD )$
over
$U$.
Then for any exceptional divisor
$E$
over
$X$ we get the inequality
\begin{equation}\label{eq:negativity_lemma_for_b-discrepancy}
 b(E;  X, \bfD ) \le b(E;  Y, \bfD ).
\end{equation}
If
$C_XE$
or
$C_YE$
is contained in the exceptional locus of
$\varphi$
or
$\varphi^{-1}$,
then \eqref{eq:negativity_lemma_for_b-discrepancy} becomes a strict inequality.
\end{lemma}

\begin{proof}
It follows from
\eqref{eq:b-discrepancy_vs_usual_discrepancy}
that
\begin{equation*}
  \sch{}{\frac{b ( E; Y, \bfD ) - b( E; X, \bfD )}{r _{ E }}} 
  = b' ( E; Y, \bfD ) - b'(E; X, \bfD ) = a ( E; Y, \bfD _{ Y } ) - a( E; X, \bfD _{
    X } ). 
\end{equation*}
Therefore the conclusions follow from the negativity lemma
for usual discrepancies \cite[Lemma 3.38]{MR1658959}.
\end{proof}

\begin{corollary}\label{MMP_preserves_bterminal_etc}
The notion of b-terminality (resp. b-canonicity, b-log terminality, b-klt, b-log canonicity)
is preserved under 
b-MMP.
\end{corollary}

\begin{proof}
Since the arguments are essentially the same, we only discuss the case of
b-terminality. Let
$
 ( X, \bfD )
$
be a b-terminal pair and consider a step of b-MMP
$
 \varphi \colon X \dasharrow Y
$.
If
$
 E
$
is an exceptional divisor over
$
 X
$,
then
\pref{lm:negativity_lemma_for_b-discrepancy}
implies
\[
 b(E;  Y, \bfD ) \ge  b(E;  X, \bfD ) > 0.
\]

If
$
 \varphi
$
is a divisorial contraction which contracts the prime divisor
$
 E \subset X
$,
then since
$
 b(E;  X, \bfD ) = 0
$
by the definitions,
we can use the second claim of
\pref{lm:negativity_lemma_for_b-discrepancy}
to see
\[
 b(E;  Y, \bfD ) >  b(E;  X, \bfD ) = 0.
\]
\end{proof}

\begin{example}\label{eg:log_MMP_as_b-MMP}
  Recall that the b-discrepancy of the proper transform b-divisor $(X,\widehat{D})$ is the same as the usual discrepancy as discussed in Examples~\ref{proper_transform_same_discrepancy},~\ref{proper_transform_same_singularites}.  In addition, the contractions and flips of the b-log MMP are simply those of the log MMP, so running the b-log MMP for the b-\logpair $(X,\widehat{D})$ is identical to running the log MMP for the pair $(X,D)$.  However, in the b-log MMP you may resolve singularities before running the MMP and 
  we note that if you first resolve $(X,\widehat{D})$ to a model which is terminal, canonical, lt or lc, the minimal model (if it exists) will have the same type of singularities by Corollary~\ref{MMP_preserves_bterminal_etc}.
\end{example}

\noindent
Next, we explain how the $G$-equivariant MMP is a special case of the b-log MMP.
\begin{example}\label{eg:equivariant_MMP_as_b-MMP}
  Recall the b-\logpair associated to the equivariant setting\sch{ which is}{,}
  discussed in \pref{eg:b-divisor_of_nonabelian_H1}.  Let $ E \subset Y $ be a prime
  $ f $-exceptional divisor and set
  $ \Etilde := \lb \pi _{ Y } ^{ - 1 } ( E ) \subset \Ytilde \rb _{\mathrm{red}} $.
  Then it follows that $ a ( \Etilde; \Xtilde ) = b ( E; X, \bfD ) $ by the following
  computation.  Let $f:Y \to X$ be a birational morphism and let
  $\ftilde:\Ytilde \to \Xtilde$ be the corresponding map on the normalizations of $Y$
  and $X$ in the Galois cover.  Write $\pi_X:\Xtilde \to X$ and
  $\pi_Y: \Ytilde \to Y$.  We have that
 \begin{eqnarray*}
   \sum_{\tilde{E}} a(\Etilde;\Xtilde)\Etilde & = & K_\Ytilde-\ftilde^*K_\Xtilde\\
   & = & \pi_Y^*(K_Y+\bfD_Y) - \ftilde^*\pi_X^*(K_X+\bfD_X) \\
   & = & \pi_Y^*(K_Y+\bfD_Y) - \pi_Y^*f^*(K_X+\bfD_X) \\
   & = & \pi_Y^*\left(\sum_Eb'(E;X,\bfD)E\right) \\
   & = & \sum_\Etilde b'(E;X,\bfD)r_E \Etilde \\
   & = & \sum_{\Etilde} b(E;X,\bfD) \Etilde
   \end{eqnarray*}

   Since $K_{\tilde{X}} = \pi^*(K_X+\bfD_X)$ we obtain that the MMP of the pair
   $ ( X, \bfD ) $ \sch{precisely corresponds}{corresponds precisely} to the
   $ G $-equivariant MMP of $ \Xtilde $.
\end{example}

%
%

\subsection{Fundamental theorems for b-\logpairs}

In this section we establish some foundational results about the b-log \sch{Minimal
  Model Theory}{MMP} by transplanting the corresponding results \sch{for}{from (the
ordinary)} log \sch{Minimal Model Theory}{MMP}.

\begin{theorem}
Let
$
 \pi \colon ( X, \bfD ) \to U
$
be a b-lc pair over $U$. If
$
 K _{ X } + \bfD _{ X }
$
is not nef, then there exists an extremal contraction.
If it is a flipping contraction, then the flip exists.
\end{theorem}

\begin{proof}
Since
$
 ( X, \bfD _{ X } )
$
is log canonical, the assertions immediately follow from
\cite[Theorem 1.19]{MR3238112} and
\cite[Corollary 1.2]{MR2929730}, respectively.
\end{proof}

\begin{theorem}\label{MMPworks}
Let
$
 ( X, \bfD )
$
be a fractional b-canonical pair.
Suppose either the dimension of $X$ is at most 3, or
$
 \bfD _{ X }
$
is big and
$
 K _{ X } + \bfD _{ X }
$
is pseudo-effective. Then
the pair admits a minimal model and the log-canonical divisor of the minimal model is semi-ample.
\end{theorem}

\begin{proof}
Since the \logpair
$
 ( X, \bfD _{ X } )
$
is klt, the assertions immediately follow from the results of
\cite{MR2601039}
and
\cite{MR1284817}.
\end{proof}

\begin{remark}\label{canonicalModelExists}
Let
$
 ( X, \bfD )
$
be a fractional b-canonical pair.
    If we assume that the log MMP terminates for $(X,\bfD_X)$
    then the pair admits a minimal model.  If the pair admits a canonical model then it is unique by \cite[Theorem 3.52]{MR1658959}.
\end{remark}

\begin{definition}
Let
$
 \varphi \colon X \dasharrow X'
$
be a birational map between normal varieties. A \emph{common resolution of}
$
 \varphi
$
is a smooth variety
$
 W
$
and projective birational morphisms
$
 p, p' \colon W \rightrightarrows X, X'
$
such that
$
 p' = \varphi \circ p
$
\sch{}{as rational maps}
(see \pref{fg:common_resolution}).  Note that one can be obtained by resolving the singularities of the closure of the graph of $\phi$ in $X \times X'$.
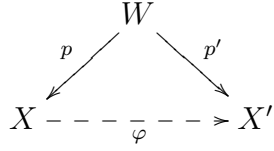
\begin{figure}[h]
\begin{minipage}{.49 \linewidth}
\begin{align*}
\xymatrix{
& W \ar[ld] _{ p } \ar[rd] ^{ p ' }& \\
X \ar@{-->}[rr] _{ \varphi } & & X'\\
}
\end{align*}
\caption{common resolution}
\label{fg:common_resolution}
\end{minipage}
\end{figure}

\end{definition}

\begin{proposition}\label{pr:b-minimal_models_are_small_birational}
Let
$
 ( X, \bfD )
$
and
$
 ( X', \bfD )
$
be b-terminal minimal models over $U$
of the same b-terminal pair.
Then any birational map
$
 \varphi \colon X \dasharrow X'
$
over
$
 U
$
is an isomorphism in codimension $1$.
\end{proposition}

\begin{proof}
The proof below is a slight modification of the one in \cite[p. 420]{MR2426353}, but we give more details
for the convenience of he readers.
Take a common resolution of singularities
$
 W
$
as in \pref{fg:common_resolution}.
By the symmetry, it is enough to show that any $p$-exceptional divisor is also $p'$-exceptional.
Consider the canonical bundle formula
\begin{equation}\label{eq:the_canonical_bundle_formula}
 K _{ W } + \bfD _{ W }
 =
 p ^{ * } ( K _{ X } + \bfD _{ X } ) + E
 =
 ( p' ) ^{ * } ( K _{ X' } + \bfD _{ X' } ) + E'.
\end{equation}
Set
\begin{align}
 F = \min ( E, E ' )
\end{align}
and
\begin{align}
  E = \Ebar + F,\\
 E ' = \Ebar ' + F.
\end{align}
The assumption is equivalent to
$
 \Ebar \neq 0
$, since
$
 ( X, \bfD )
$
is b-terminal.
By \pref{lm:HIT} below, one can find an irreducible curve
$
 C
$
such that
$
 C \not\subset \Supp \Ebar '
$,
$
 ( \Ebar . C ) < 0
$,
and
$
 p ( C ) = \mbox{point}
$.
This clearly contradicts the equality
\eqref{eq:the_canonical_bundle_formula}, since
\begin{equation}
 0 > \lb  p ^{ * } ( K _{ X } + \bfD _{ X } ) + \Ebar ) \rb . C
 =
 \lb  ( p' ) ^{ * } ( K _{ X' } + \bfD _{ X' } ) + \Ebar ' \rb . C
 \ge 0.
\end{equation}
\end{proof}

\begin{lemma}\label{lm:HIT}
Let
$
 p \colon W \to X
$
be a birational projective morphism of normal varieties over a field of characteristic zero.
Assume that $W$ is smooth, and let
$
 \Ebar
$
be a non-trivial effective $p$-exceptional $\bQ$-divisor and
$
 \Ebar '
$
be an effective divisor on $W$ none of whose component is contracted by $p$.
Then there exists an irreducible projective curve $C \subset W$ contracted to a point by $p$,
$
 \Ebar . C < 0,
$
and
$
 C \not\subset \Supp \Ebar '
$.
\end{lemma}

\begin{proof}
The proof below is taken from \cite[Proof of 3.39]{MR1658959}.
Consider the decomposition
\begin{align}
 \Ebar = \sum _{ i = 2 } ^{ \dim X } \Ebar _{ i },
\end{align}
where
$
 \Ebar _{ i }
$
is the sum of the components $\Gamma \subset \Ebar$ such that the codimension of
$
 p \lb \Gamma \rb
$
is $i$. Let $k \ge 2$ be the minimum integer such that
$
 \Ebar _{ k } \neq 0
$.
Take a general complete intersection
$
 H ^{ k - 2 }
$
of codimension
$
 k - 2
$
\emph{on $W$}. Set
$
 \Gbar \coloneqq \Ebar | _{ H ^{ k - 2 } }
$.
Then by the genericity we may assume that
$
 \Gbar _{ i } = \Ebar _{ i + k - 2 } | _{ H ^{ k - 2 } }
$
for all
$
 i \ge 2
$
and that no irreducible component of
$
 \Ebar ' \cap H ^{ k - 2 }
$
is contracted by the morphism
$
 p | _{ H ^{ k - 2 } }
$.
We may also assume that if we let
$
 \Hbar ^{ k - 2 }
$
be the normalization of
$
 p \lb H ^{ k - 2 } \rb
$,
then the morphism
$
 p | _{ H ^{ k - 2 } } \colon H ^{ k - 2 } \to \Hbar ^{ k - 2 }
$
is projective and birational.
Note that if one can find an irreducible projective curve
$
 C \subset H ^{ k - 2 }
$
which is contracted by
$
 p | _{ H ^{ k - 2 } }
$,
$
 \Gbar . C < 0
$,
and
$
 C \not \subset H ^{ k - 2 } \cap \Ebar '
$,
then as a curve on $W$ it has the required properties as well. Hence we can assume that $k = 2$.

If $k = 2$ take a general complete intersection $S \subset X$ of dimension 2 which is normal
\cite[Theorem 7]{MR0037548}, $T = p ^{ - 1 } \lb S \rb \subset Y$ is smooth, and
$
 T \cap \Ebar _{ 2 } = T \cap \Ebar
$.
We may moreover assume that
$
 p | _{ T } \colon T \to S
$
is an isomorphism on an open neighbourhood of
$
 T \cap \Supp \Ebar '
$,
since the image of the exceptional locus of
$
 \Supp \Ebar '
$
under the morphism
$
 p
$
has codimension at least three. Hence it follows that
$
 N \coloneqq T \cap \Ebar _{ 2 }
$
is a non-trivial effective
$
 p | _{ T }
$-exceptional divisor none of whose irreducible component is contained in $\Supp \Ebar '$.
By the Hodge index theorem
$
 N ^{ 2 } < 0
$. Since $N$ is an effective divisor, there is at least one component
$
 C \subset N
$
such that
$
 N . C < 0
$.
It is now obvious that the curve $C$, seen as a curve on $W$, has all the required properties.
\end{proof}

We next look at some results that hold specifically for surfaces.  A b-terminal pair $(S,\bfD)$, where $S$ is a surface will have $(S,\bfD_S)$ log terminal, so we know that $S$ has quotient singularities.

\begin{corollary}
Let
$
 ( S, \bfD )
$
be a b-terminal pair of dimension $2$ with non-negative Kodaira dimension.
Then it admits a unique minimal model.
\end{corollary}

\begin{proof}
The existence of a minimal model is already settled.
The uniqueness follows from the previous proposition and the following
well-known lemma.
\end{proof}

\begin{lemma}
  If a birational map $ \varphi \colon S \dasharrow S' $ between normal surfaces is
  an isomorphism in codimension $1$ \sch{}{on both $S$ and $S'$}, then it is an
  isomorphism.
\end{lemma}

\begin{proof}
Consider a common resolution
$
 p, p' \colon W \to S, S'
$
satisfying
$
 p' = \varphi \circ p 
$.
By the assumption, an irreducible curve
$
 C \subset W
$
is contracted to a point by
$
 p
$
if and only if it is contracted to a point by
$
 p '
$.
Therefore, if we consider the image
$
 \Gamma
$
of the morphism
$
 p \times p' \colon W \to S \times S'
$,
 the natural projections
$
 \Gamma \to S
$
and
$
 \Gamma \to S '
$
are birational and finite, hence isomorphisms by the Zariski's main theorem
\cite[Chapter III, Corollary 11.4]{Hartshorne}.
Thus
$
 \varphi
$
extends to \sch{the}{an} isomorphism whose graph is
$
 \Gamma
$.
\end{proof}

\begin{theorem}\label{th:flop_connect_b-minimal_models}
  Under the assumptions of \pref{pr:b-minimal_models_are_small_birational}, the
  birational map $ \varphi $ \sch{is}{can be} decomposed into a sequence of flops.
\end{theorem}

\begin{proof}
\pref{pr:b-minimal_models_are_small_birational} gives the only required modification of the proof of \cite[Theorem 1]{MR2426353}.
\end{proof}


\section{Toroidal b-\logpairs}\label{sec:toroidal_blog_pairs}

In this section we will discuss discrepancy and b-terminalizations for toroidal
b-\logpairs.  We will repeat some earlier results, but we include new proofs using
toroidal methods, since they are more explicit and constructive.  \Logsmooth
\sch{varieties}{pairs} are toroidal and toroidal varieties are naturally stratified,
so we begin by \sch{showing}{proving} some results concerning discrepancy for
\logsmooth pairs.

\subsection{\Logsmooth stratifications}
Given \alogsmooth log canonical pair $(X, D)$, we obtain a stratification of $X$
defined as follows: when the support of $D$ is given by $\bigcup D_i$ for irreducible
divisors $D_i$ a stratum is defined to be a irreducible component of an intersection
of some of the divisors $D_i$.  When one blows up a stratum $Z\subset X$ to obtain
$ f \colon Y = \Bl _{ Z } X \to X $, we define $D_Y$ by the equality of
$ \bQ $-divisors
\begin{align}
K_Y + D_Y
=
f ^{ * } ( K_X + D ).
\end{align}
Since
$
 ( Y, D _{ Y } )
$
is again \alogsmooth pair, we obtain a stratification of $Y$.
We repeat this process to define the boundary divisor for any model over
$X$
obtained by repeatedly blowing up strata.
\begin{proposition}\label{pr:discrepancy_of_other_divisors}
Let $(X, D)$ be \alogsmooth klt pair and
$E$ an exceptional divisor over
$X$ which cannot be obtained by
repeatedly blowing up the strata. Then
\begin{align}
a(E; X, D) > 0.
\end{align}
\end{proposition}
The following lemma will be used in the proof of
\pref{pr:discrepancy_of_other_divisors}.
\begin{lemma}\label{lm:transitivity_of_discrepancy}
Let
$f \colon Y \to X$ and
$g \colon Z \to Y$
be birational morphisms between
normal projective varieties.
Let
$D$
be a $\bQ$-divisor on
$X$
such that
$K_X+D$
and
$K_Y+D_Y$,
where
$D_Y:=f^{-1}_*D$,
are both
$\bQ$-Cartier.
Let
$
 E \subset Y
$
be an
$f$-exceptional divisor, and
$F\subset Z$ be a
$g$-exceptional divisor
which satisfies
$
 C _{ Y } F
 \subset E
$.
Assume that for any
$f$-exceptional divisor
$
 E' \subset Y
$
other than $E$ we have
$
 a ( E'; X, D )
 \ge 0
$.
Then
\begin{align}\label{eq:transitivity_of_discrepancy}
a ( F; X, D )
\ge
a ( F; Y, D _{ Y } )
+ a ( E; X, D ).
\end{align}
\end{lemma}

\begin{proof}
Define the divisor
$D'$
on
$Y$ by the equality
\begin{align}
 K _{ Y } + D'
 =
 f ^{ * } ( K _{ X } + D ).
\end{align}
We see
\begin{align}
 a ( F; X, D ) - a ( F; Y, D _{ Y } )
 = a ( F; Y, D' ) - a ( F; Y, D _{ Y } )\\
 = \sum _{ E' } a ( E'; X, D ) m _{ F } ( E' )
 + a ( E; X, D ) m _{ F } ( E )
 \ge a ( E; X, D ),
\end{align}
concluding the proof.
\end{proof}

\begin{proof}[Proof of \pref{pr:discrepancy_of_other_divisors}]
Any exceptional divisor
$E$ over
$X$ can be realized as
a codimension $1$ regular point
on a variety, by
repeatedly blowing up its centre for finitely many times
(starting with the blow-up of
$C_X E$); see
\cite[2.45]{MR1658959}.
Let
$t$ be the number of necessary blowups.
We prove the statement by induction on $t$.

Suppose $t=1$.
By replacing
$X$
with
$
 X \setminus \Sing ( C_X E )
$,
we may assume
$C_X E$ is smooth.
Set
$
 c = \codim _{ X } C _{ X } E
 $,
 and let $D=\sum a_iD_i$ be the decomposition of $D$ into irreducible divisors.
 Reorder the $D_i$ so that $
 C _{ X } E
 \subset D_i
 \iff
 i \le d
$.  Note that $\mult_{C_XE}D_i = 1$ for $i \leq d$ and since $C_XE$ is not a stratum, we have the inequality
$
 c > d
$.
So we obtain the formula
\begin{align}
 a ( E; X, D )
 =
 c - 1 - \sum _{ i = 1 } ^{ d } a _{ i }.
\end{align}

If $d=0$, we see
$
 a ( E; X, D ) = c - 1 \ge 1
$.
If $d>0$, we see
\begin{align}
c - 1 - \sum _{ i = 1} ^{ d } a _{ i }
=
c - d - 1 + \sum _{ i = 1 } ^{ d } ( 1 - a _{ i } )
\ge 0 + ( 1 - \min\{ a_{ i }\} ) > 0.
\end{align}

Now let us consider the induction step.
Consider the sequence of blowups
which realizes the divisor $E$:
\begin{equation}
 X_t \to X_{t-1} \to \cdots \to X_1 \to X
\end{equation}
Let
$
 E _{ i } \subset X _{ i }
$
be the exceptional divisor of the
$i$-th blowup.

Suppose that
$C_X E$
is not a stratum.
Then by setting
$Z= X_t$
and
$Y=X_1$,
we can apply 
\pref{lm:transitivity_of_discrepancy}
to obtain the inequality
\begin{align}\label{eq:discrepancy_of_non-toric_exceptional_divisor}
 a(E; X, D)\ge a(E;Y, D_Y)+ a(E_1; X, D).
\end{align}

Note that by the case
$t=1$, we know
$
 a ( E _{ 1 }; X, D ) > 0
$.
Moreover, since
$
 D _{ Y }
$,
the strict transform of
$
 D
$ on
$Y$,
is again \sch{SNC}{snc} and
$
 C _{ Y } E
$
is not a stratum, we can apply the induction hypothesis
to see
$
 a ( E; Y, D _{ Y } ) > 0
$.
Thus we obtain the conclusion from
\eqref{eq:discrepancy_of_non-toric_exceptional_divisor}.

Finally, suppose that
$
 C _{ X } E
$
is a stratum. In this case, define the divisor
$
 D'
$
by
$
 K _{ Y } + D'
 =
 f ^{ * } ( K _{ X } + D )
$.
By applying the induction hypothesis to
$
 ( Y, D' )
$, we get
\begin{align}
 a ( E; X, D ) = a ( E; Y, D' ) > 0.
\end{align}
Thus we conclude the proof.
\end{proof}

\begin{theorem}\label{th:positivity_of_Brauer_discrepancy_of_non_toric_divisors}
Let
$( X, \bfD )$
be \alogsmooth fractional b-\logpair.
Let
$E$
be any exceptional divisor over
$X$
with
$b(E;  X, \bfD )\le 0$.
Then
$E$
is obtained by repeatedly
blowing up strata. 
\end{theorem}
\begin{proof}
  Since the \logpair $( X, \bfD _{ X })$ is \logsmooth and $\bfD_X$ is fractional,
  \cite[Corollary 2.31(3)]{MR1658959} shows that it is klt.  So for any exceptional
  divisor $E$ which is not obtained by repeatedly blowing up strata, we have
\begin{align}
b'(E;  X, \bfD )\ge a(E;  X, \bfD _{ X } )>0
\end{align}
by
\pref{rm:comparison_of_discrepancies}
and
\pref{pr:discrepancy_of_other_divisors}.
Thus we conclude the proof.
\end{proof}

%
%

\subsection{Toric \texorpdfstring{\MakeLowercase{b}}{b}-\logpairs}

%
%

Now we will study b-\logpairs where the b-divisor is supported on a toric divisor in a toric variety.  In addition to allowing explicit computations, we will provide a second proof of one of the main results of this paper, Theorem~\ref{th:existence_of_Brauer_terminal_resolution}, which shows the existence of b-terminal resolutions, or b-terminalizations.  This result is of central importance, since without it,
one can not begin the b-log minimal model program with b-terminal singularities.

Let us review some basic facts from toric geometry.
We will use results and notation from~\cite{MR2810322}.
Let $X$ be a toric variety with open 
dense torus $T\simeq \Gm^n \subseteq X$.  The variety $X$ 
is determined
by a rational fan $\Sigma$ in the real vector space spanned by the lattice  $N=\Hom(\Gm,T)$.
In particular toric geometric valuations of $\bk (T)$ are given by rational rays in $\R N$.
More precisely, there is a correspondence 
between primitive vectors $w=(w_1,\ldots,w_n) \in \Z^n \simeq N$
(primitive means $\gcd(w_i)=1$)
and toric divisors $D_w$ on some toric model of $\bk (T)$.

Let $\Sigma(1)$ be the set of rays of the fan $\Sigma$.
Write $u_\rho$ for the minimal generator in $N$ of a ray $\rho$ in $\Sigma$.
Write $D_\rho$ or $D_{u_\rho}$ for the divisor associated to $\rho$.
A toric $\Q$-divisor $D$ can be written as
$D=\sum d_\rho D_\rho$.
If $D$ is $\Q$-Cartier it has an associated \emph{support function}
$\phi_D \colon |\Sigma| \to \R$ with the following properties:
\begin{enumerate}
\item $\phi_D$ is linear on each cone $\sigma \in \Sigma$.
\item $\phi_D(N) \subseteq \Q$.
\item $\phi_D(u_\rho)=-d_\rho$.
\item $D=-\sum_{\rho \in \Sigma(1)} \phi_D(u_\rho) D_\rho$.
\end{enumerate}

Let $X$ be the toric variety associated to a simplicial fan $\Sigma$.
Recall the following equality:
\begin{equation}
 K_X = \sum_{\rho \in \Sigma(1)} -D_\rho
\end{equation}
Note also that support functions are preserved by pullback.  More precisely, let $f:\widetilde{X} \to X$
be the map of toric varieties associated to a map of fans $f_\Sigma:\widetilde{\Sigma} \to \Sigma$.  Then for a
$\Q$-Cartier divisor $D$ on $X$, we get that 
\[\phi_{f^*D}=\phi_D \circ f_\Sigma : |\widetilde{\Sigma}| \to \R.\]

The proof of the following proposition follows the notation and proof of \cite[Proposition 11.4.24]{MR2810322}.
Given a b-divisor $\bfD  \in \bDiv(K)$ and a normal model $X$ of $K$, as in Definition~\ref{ramification:index:defn}, we will write 
\[ \bfD_X = \sum_{\Gamma} d_\Gamma \Gamma = \sum_{\Gamma} \lb 1 -\frac{1}{r_\Gamma}\rb \Gamma.\]
We say that $(X,\bfD)$ is a toric b-\logpair if $X$ is a toric variety and $\bfD$ is supported on toric divisors for all models.

\begin{proposition}\label{prop:toric_descrpancy_formula} 
Let $X$ be the toric variety associated to the fan $\Sigma$ and let
 $(X,\bfD)$ be a toric b-\logpair.  Let $w$ be
an element of the lattice $N$ that is in a (possibly not maximal) simplicial cone $\sigma$ in $\Sigma$,
with ramification index $r_w$.  Suppose $\sigma$ has a set of minimal generators
$\sigma=\langle v_1,\ldots,v_m \rangle$ with
$v_1,\ldots,v_m$ in the lattice $N$ with
ramification indices $r_1,\ldots,r_m.$  Write
$w=a_1v_1+\cdots+a_mv_m$.
Then the discrepancy of the divisor $D_w$ associated to $w$, over $(X,\bfD)$ is given by
\begin{equation}
 b'( D_w; X, \bfD )
 =
 \frac{a_1}{r_1} + \cdots + \frac{a_m}{r_m} -\frac{1}{r_w}.
\end{equation}
\begin{equation}\label{eq:formula_for_toric_b_discrepancy}
 b( D_w; X, \bfD )
 =
 \frac{a_1r_w}{r_1} + \cdots + \frac{a_mr_w}{r_m} -1.
\end{equation}
\end{proposition}
\begin{proof}
  Let $\widetilde{\Sigma}$ be a simplicial refinement of $\Sigma$ that contains
  $\langle w \rangle$ as a ray.  Let $E=\widetilde{\Sigma}(1) \setminus \Sigma(1)$ be
  the set of exceptional divisors.  Note that
  \begin{align*}
    K_{X}+\bfD_X &= -\sum_{\rho \in \Sigma(1)} D_\rho + \sum_{\rho \in \Sigma(1)}\lb 1
    - \frac{1}{r_\rho}\rb D_\rho \\
    &= -\sum_{\rho \in \Sigma(1)}\lb\frac{1}{r_\rho}\rb D_\rho
  \end{align*}
  and similarly for $\widetilde{X}$.  Now
  \begin{align*}
    \sum_{\rho \in E} b'(D_\rho,X,\bfD)D_\rho 
    &= K_{\widetilde{X}}+\bfD_{\widetilde{X}}- f^*(K_{X}+\bfD_X) \\ 
    &= -\sum_{\rho \in
      \widetilde{\Sigma}(1)}\lb\frac{1}{r_\rho}\rb D_\rho +f^*\lb\sum_{\rho \in
      \Sigma(1)}\lb\frac{1}{r_\rho}\rb      D_\rho\rb. 
  \end{align*}
  Now let $\phi$ be the support function associated to the divisor
  \[B=\sum_{\rho \in \Sigma(1)}\lb\frac{1}{r_\rho}\rb D_\rho.\]
  Since $|\widetilde{\Sigma}|=|\Sigma|$ we will also denote the support function of
  $f^*B$ by $\phi$.  Note that
\begin{equation}
\begin{split}
  f^*B = - \sum_{\rho \in \widetilde{\Sigma}(1)} \phi(u_\rho)D_\rho
  =  -\sum_{\rho \in \Sigma(1)}  \phi(u_\rho)D_\rho - \sum_{\rho \in E} \phi(u_\rho) D_\rho\\
  = \sum_{\rho \in \Sigma(1)} \frac{1}{r_\rho}D_\rho - \sum_{\rho \in E} \phi(u_\rho)
  D_\rho,
\end{split}
\end{equation}
so that
\begin{equation}
 \sum_{\rho \in E} b'(D_\rho,X,\bfD)D_\rho
 =
 -\sum_{\rho \in E}\lb \frac{1}{r_\rho}+\phi(u_\rho)\rb D_\rho.
\end{equation}

To elucidate this sum we consider the coefficient of $D_w$ as above.
Since $w= a_1v_1+\cdots+a_m v_m$ and $\phi(v_i)=-1/r_i$ we get the desired result.

\end{proof}

The next result follows from Remark~\ref{rm:finiteness_of_non_terminal_witnesses}, since toric b-\logpairs are b-klt.  However, since this also follows directly from toric geometry, we include an alternate proof.

\begin{proposition}\label{prop:finitely_many_toric}
Let $(X,\bfD)$ be a toric b-\logpair with $X$ normal and $\Q$-factorial.  Then there are finitely many toric divisors $D_w$ over $X$ such that $b(D_w,X,\bfD) \leq 0$. 
\end{proposition}
\begin{proof}
 
Let $\Sigma$ be the fan associated to $X$, let $\sigma \in \Sigma$ be a cone, and write $\sigma=\la v_1,\ldots,v_m \rangle$ for minimal generators $v_i$ in $N$.  Given a primitive vector $w = a_1v_1+\cdots+a_mv_m \in \sigma,$ the formula \eqref{eq:formula_for_toric_b_discrepancy} for discrepancy gives 
\[ b(D_w;X,\bfD)=\frac{a_1r_w}{r_1} + \cdots + \frac{a_mr_w}{r_m} -{1}.\]
Since  $r_w \geq 1$, and
this is positive when the $a_i$ are sufficiently large, there are finitely possible $w \in \sigma$ such that 
$b'(D_w;X,\bfD) \leq 0$ is finite.  The result follows since there are finitely many cones in $\Sigma$.
\end{proof}

\begin{proposition}
Let $X$ be a toric variety associated to the fan $\Sigma$.  Let $D_{w_1},\ldots,D_{w_p}$ be a finite set of divisors over $X$ corresponding to the primitive vectors $w_1,\ldots,w_p \in |\Sigma|$ with $\la w_i \rangle \notin \Sigma$.
Then there is a $\Q$-factorial toric variety $X'$ and a birational proper toric morphism $f:X' \to X$ such that the exceptional divisors of $f$ are exactly the toric divisors $D_{w_1},\ldots,D_{w_p}$.
\end{proposition}
\begin{proof}
We first replace $\Sigma$ by a simplicial refinement by triangulating the non-simplicial cones as described in
\cite[Proposition 11.1.7]{MR2810322}.
Given a divisor $D_{w_1}$ to extract, we form the star subdivision 
$\Sigma_1:=\Sigma^*(w_1)$ as constructed in \cite[p.~515, Section~11.1]{MR2810322}.  This forms a simplicial refinement of $\Sigma$ with exactly one new ray
$\langle w_1\rangle$.  We then repeat for all $w_i$ to obtain $\Sigma'$ simplicial with new rays corresponding exactly to those primitive vectors with non-positive discrepancy. 
\end{proof}

The next result follows from the more general Theorem~\ref{th:existence_of_Brauer_terminal_resolution}, but we will use the statement below in the toroidal setting to provide a more constructive proof of this theorem.

\begin{proposition}\label{prop:exists_bterm}
Any toric b-\logpair has a toric $\Q$-factorial b-terminalization.
\end{proposition}
\begin{proof}
  We first pass to a toric \logres $(X,\bfD)$ of our given toric b-\logpair.  This is done by simply finding a toric resolution of singularities since the toric divisor will automatically be simple normal crossings.  Then we obtain $f:X' \to X$ 
with simplicial fan $\Sigma'$ refining $\Sigma$,
which extracts exactly the toric divisors $D_w$ such that $b(D_w;X,\bfD)\leq 0$ using the above two propositions.
Now by Lemma~\ref{lm:brauer_discrepancy_under_extraction_of_bad_divisor}, the discrepancy for $X'$ is larger than the
discrepancy for $X$.
\end{proof}

\subsection{Toroidal b-\logpairs}

Now we consider toroidal b-\logpairs.  We will use \cite{MR0335518} for definitions, notation and basic results, but we will provide some heuristic explanations.  We say a \logpair $(X,D)$ is toroidal if $U=X\setminus \Supp D \subset X$ is a toroidal embedding as in
\cite[p.54]{MR0335518}.
As explained in \cite[p.71]{MR0335518}, we can associate a {\it conical polyhedral complex with integral structure}
$
 \Delta = (|\Delta|,\sigma^Y, M^Y)
$
 to a toroidal embedding.  A conical polyhedral complex consists of a finite collection of cones $\{\sigma^Y\}_Y$ with an integral structure $\sigma^Y \subset \R N^Y = \R \Hom(M^Y,\Z)$ for each cone, indexed by the natural stratification $\{Y\}$ associated to the toroidal embedding $U \subset X$.  The affine toric variety associated to the cone $\sigma_Y$ corresponding to stratum $Y$ describes the
 \'etale local structure of $U \subseteq X$ at the generic point $y$ in $Y$.
 A face of a cone in $\Delta$ is again a cone in $\Delta$ and the cones (with their integral structures) are glued along faces.  Unlike the case of a fan used in toric geometry, there is no ambient lattice $N$ so that $|\Delta| \subset \R N$ and we can have more than two faces glued along a face of codimension one.  The conical polyhedral complex does not uniquely determine the toroidal variety.  However, akin to refinements of fans in toric geometry, 
 there is a correspondence between {\it finite rational partial polyhedral (f.r.p.p.) decompositions}  (see \cite[Definition 2, p.86]{MR0335518} for the precise definition) $\Delta'$ of $\Delta$ with $|\Delta'| = |\Delta|$ and proper birational morphisms $X' \to X$ that are {\it allowable (or toroidal) modifications} by \cite[Theorem $6 ^{ * }$, p.90]{MR0335518}.  
 We say a b-\logpair $(X,\bfD)$ is {\it toroidal} if $(X,\bfD_X)$ is toroidal.
 Note that this implies that for all allowable modifications, $X' \to X$, we have that $(X,\bfD_{X'})$ is also toroidal.  
 We call the exceptional divisors that are divisors in allowable modifications
 the {\it toroidal divisors} over $X$.

 Since a toroidal variety is characterized by being \'etale locally an affine toric
 variety with toric boundary, we see that \alogsmooth pair $(X,D)$ is toroidal.
 Suppose we are given a rational ray $\rho$ in a cone $\sigma_Y \subseteq \Delta$
 corresponding to the stratum $Y$.  We can form the star subdivison $\Delta^*(\rho)$
 where we add one ray $\rho =\langle u_\rho\rangle$ and subdivide every cone
 $\tau = \langle u_1,\ldots, u_m \rangle$ containing $\rho$ by forming cones
 $\langle u_\rho, u_1,\ldots,\hat{u}_i,\ldots,u_m\rangle$ exactly as in \cite[p.~515,
 Section~11.1]{MR2810322}.  We note the following facts about the star subdivisions:
 \begin{itemize}
 \item $\Delta^*(\rho)$ is a f.r.p.p. decompostion of $\Delta$.
  \item $|\Delta^*(\rho)| = |\Delta|$.
\item There is the corresponding projective allowable modification $X^*(\rho) \to X$.
  \item If $\Delta$ is simplicial, then so is $\Delta^*(\rho)$ and $X^*(\rho)$ is $\Q$-factorial.
 \item The cones of dimension one (rays) in $\Delta^*(\rho)$ are  rays in
   $\Delta$ with the addition of the one new ray $\rho.$
\item The divisor $D_\rho$ is Cartier.
 \end{itemize}

 The next lemma follows easily from  \cite[Theorem $10 ^{ * }$, p.90]{MR0335518}.
   \begin{lemma}  Let $(X,D)$ be toroidal and let $\Delta = (|\Delta|,\sigma^Y,M^Y)$ be the associated conical polyhedral complex with integral structure.
     Let $\sigma_Y = \langle u_1,\ldots,u_m\rangle$ be minimal primitive generators of $\sigma_Y$ and $\rho = \langle \sum u_i \rangle$
   then $X^*(\rho) \to X$
   is the normalization of the
   blow up of the stratum $Y$.
\end{lemma}
   \begin{proof}
     Since the ideal sheaf of any stratum of the \logpair
$
 ( X, D) 
$
is a canonical coherent sheaf of fractional ideals in the sense of \cite[p.90]{MR0335518},
any blowup of
$
 X
$
along a stratum corresponds to a f.r.p.p. decomposition  of
$
 \Delta
$
by \cite[Theorem $10 ^{ * }$, p. 93]{MR0335518}.  This is clearly given by the star subdivision described above.
\end{proof}
   \begin{corollary}\label{cor:toroidal_blow_up_strata}
     If $(X,D)$ is \logsmooth, then the divisors over $X$ obtained by blowing up strata are exactly the toroidal divisors over $X$.
     \end{corollary}
   \begin{proof}  It is clear that the divisors obtained by blowing up strata will be toroidal, so we must prove the converse.
     We first consider the case of affine space as a toric variety $\bG_m^n \subset \bA^n$.  Let $e_1,\ldots,e_n$ be the minimal generators of the cone in the lattice $N$.  In this case, a toric divisor corresponds to primitive vector $w$ with all coordinates non-negative.  We will construct a sequence of blow ups at smooth toric subvarities to obtain $\langle w \rangle$ as a ray.  We write $w=\sum a_ie_i$ with $a_i \geq 0$.  If all $a_i \leq 1$ we are done.  Otherwise we form the star subdivision at $v=\sum \sgn (a_i) e_i$ where $\sgn(a_i)$ is the sign function.  Now $w$ will be in a new smooth simplicial cone which includes $v$ as a vertex and we have
     \[w=v+\sum_{a_i \neq 0} (a_i -1) e_i.\]  Now the coefficients of $w$ in terms of generators of the new cone are smaller and the coefficient of $v$ is one.  So by repeating this process we will eventually obtain $w$.

     Now given a general \logsmooth toroidal pair $(X,D)$, any toroidal divisor corresponds to a ray $\rho =\langle w \rangle$ in some cone $\sigma_Y$ associated to some stratum $Y$.  Since the cone $\sigma_Y$ is smooth and simplicial, we can carry out the sequence of star subdivisions described above.  This will yield a sequence of blow ups at strata eventually realizing the toroidal divisor corresponding to $w$.
     \end{proof}

   \begin{proposition}
     Let $(X,\bfD)$ be \alogsmooth b-\logpair.  Let $E$ be a divisor over $X$.  If $E$ is not toroidal then $b(E;X,D) >0$.  If $E$ is toroidal then the centre of $E$ on $X$ is in a strata $Y$ and $E$ corresponds to a ray $\langle w \rangle$ in the cone $\sigma_Y = \langle v_1,\ldots,v_m \rangle$ with minimal generators $v_i$.  Let $r_w$ be the ramification index of $E$,  let $r_i$ be the ramification index of $v_i$, and write $w=a_1v_1 + \cdots + a_mv_m$.  Then
     $$ b(E;X,\bfD)=\frac{a_1r_w}{r_1} + \cdots + \frac{a_mr_w}{r_m} -{1}.$$
   \end{proposition}
   \begin{proof}
     This first statement follows by combining Corollary~\ref{cor:toroidal_blow_up_strata} and Theorem~\ref{th:positivity_of_Brauer_discrepancy_of_non_toric_divisors}.  For the second statement, we know by Corollary~\ref{cor:toroidal_blow_up_strata} that $E$ can be obtained by blowing up strata and so will appear in a toroidal morphism that is \'etale locally toric along $Y$.  So we can apply Proposition~\ref{prop:toric_descrpancy_formula}.
     \end{proof}

\begin{proposition}\label{prop:exists:terminal:extraction:toroidal}
Let $(X,\bfD)$ be a toroidal b-\logpair.  Then there is a birational projective toroidal modification
$
 f \colon X' \to X
$
such that
$
 X'
$
is 
$
 \bQ
$-factorial and the exceptional divisors of $f$ are exactly the toroidal divisors $D_w$ over $X$ with
$
 b ( D_w; X, \bfD ) \leq 0
.$
\end{proposition}

\begin{proof}

As noted in
\pref{rm:finiteness_of_non_terminal_witnesses},
there are only finitely many exceptional divisors
over
$X$
whose b-discrepancies are non-positive.  Alternatively, this can be seen by combining Theorem~\ref{th:positivity_of_Brauer_discrepancy_of_non_toric_divisors} and Proposition~\ref{prop:finitely_many_toric}.

Let $\Delta = (|\Delta|, \sigma^Y, M^Y)$ be the conical polyhedral complex with integral structure associated to $X$ and let $S$
be the set of such divisors with non-positive b-discrepancies.  Now take any divisor
$
 E
$
in
$
 S
$.
By \pref{th:positivity_of_Brauer_discrepancy_of_non_toric_divisors},
$
 E
$
is obtained by repeatedly blowing up strata. Hence there is a f.r.p.p. decomposition of
$
 \Delta
$
in which there exists a one-dimensional cone
$
 \rho _{ E }
$
corresponding to
$
 E
$.

 So there is a vector $w \in M^Y$ in a cone $\sigma^Y$ in $\Delta$.
We take the star subdivision $\Delta(w)$ of $\Delta$, and repeat inductively for all $w \in S$,
 until we obtain $\Delta'$ a f.r.p.p decomposition of $\Delta$ whose set of one-dimensional cones is
$ \{ \rho _{ E } \mid E \in S \}$ together with those in $\Delta$.

There exists a corresponding projective allowable modification
$
 Y \to X 
$
by \cite[Theorem $6 ^{ * }$, p.90]{MR0335518}, which
extracts only those divisors which are contained in
$
 S
$. 
\end{proof}

\begin{corollary}[Proof of \pref{th:existence_of_Brauer_terminal_resolution} via toroidal modification]\label{secondproof}
Let $(X,\bfD)$ be a b-\logpair.  Then there is a projective birational map $Y \rightarrow X$ such that $(Y,\bfD)$ is $\Q$-factorial and b-terminal.
\end{corollary}
\begin{proof}  Let $(X,\bfD)$ be a b-\logpair.
By
\pref{lm:log_smooth_resolution},
we find a \logres
$X_1 \to X$
of the \logpair
$
 ( X, \bfD _{ X } )
$
so that the pair
$
 ( X_1, \bfD _{  X _{ 1 } } )
$
is \logsmooth.  Note that $(X_1,\bfD_{X_1})$ is toroidal.  Now by Proposition~\ref{prop:exists:terminal:extraction:toroidal} we can find a projective allowable modification $Y \to X_1$ that extracts 
exactly the toroidal divisors $D_w$ over $X_1$ with
$
 b ( D_w; X_1, \bfD ) \leq 0.
$
By Theorem~\ref{th:positivity_of_Brauer_discrepancy_of_non_toric_divisors} these are all the divisors over $X$
with $b ( D_w; X_1, \bfD ) \leq 0.$  So by Lemma~\ref{lm:brauer_discrepancy_under_extraction_of_bad_divisor}  we see that $(Y,\bfD)$ is b-terminal, and as in the first proof of \pref{th:existence_of_Brauer_terminal_resolution}, the composition
$
 Y \to X_1  \to X
$
is a desired b-terminalization.
\end{proof}
  
\section{Toric Brauer Classes}\label{toric_brauer_classes}

Let $X$ be a toric variety with open dense torus $T$ of dimension $n$.  We define a toric Brauer pair to
be a pair $(X,\alpha)$ where $\alpha \in \Br T \simeq \wedge^2(\Hom(\mu,\Q/\Z)^n)$ as noted in \cite{MR1085941}.
Following \cite{MR1085941}, we fix a primitive $p$-th root of unity so we have an isomorphism $\Z/p \simeq \mu_p$.  At this point, $p$ is an arbitrary non-zero integer, but we will often restrict to $p$ being prime and note when this occurs. Then we associate a skew symmetric matrix
to $\alpha$
$
 M_\alpha \in \lb \Z/p \rb ^{n \times n},
$
where $p$ is the order of $\alpha$.  Let $\rho =\Cone(w)$ be a ray in $\R N$ generated by the primitive vector $w \in N$, and let $\overline{w} \in (\Z/p)^n$ be the reduction of $w$ modulo $p$.  They also show in \cite[Lemma 1.7(b)]{MR1085941} that the Brauer class $\alpha$ ramifies on the divisor $D_{\la w \ra}$
if $M_\alpha \overline{w}$ in $(\Z/p)^n$ is non-zero, and that the ramification index of $\alpha$ on $D_{\la w \ra}$ is the order of $M_\alpha \overline{w}$ in $(\Z/p)^n$. 


\begin{proposition}\label{full_rank_toric}
  Let $(X,\alpha)$ be a toric Brauer pair such that $\alpha$ has odd prime order $p$.  If $M_\alpha$ has full rank then $(X,\alpha)$ is b-terminal (b-canonical, b-lt, b-lc) if and only if $X$ is terminal (canonical, lt, lc).
\end{proposition}
\begin{proof}
Let $w=(a_1,\ldots,a_n)$ be a primitive vector in the lattice $N$.  Since $M_\alpha$ has full rank
$M_\alpha \overline{w} =0$ if and only if $w \in pN,$ but then $w$ is not primitive.  So the order of $M_\alpha w$ is $p$ for all primitive $w$.  Then when we compute the discrepancy using the formula of \eqref{eq:formula_for_toric_b_discrepancy}, all $r_i = r_w = p$, and
\[b(D_w ; X, \bfD) = (a_1 + \cdots + a_n -1) = a(D_w , X).\]
\end{proof}
Note that in dimension two, if $(X,\alpha)$ is b-terminal then $X$ is terminal (equivalently smooth) as shown in \cite{MR2180454}.  This yields the following question.
\begin{question}
  Is there a natural condition on a b-divisor $\bfD$, or a Brauer class $\alpha$ so that $(X,\bfD)$ or $(X,\alpha)$ b-terminal implies $X$ is terminal?
\end{question}
However, when $p$ is odd, $M_\alpha$ must have even rank and so we cannot expect the above results to hold when both $n$ and $p$ are odd.
Below we present an example which shows that Proposition~\ref{full_rank_toric} cannot be generalized to hold dimension 3.
\begin{example}
We present an example of a toric Brauer pair $(X,\alpha)$ in dimension 3 such that $X$ is log terminal with the minimal discrepancy arbitrarily close to
$
 - 1
$,
whereas the b-\logpair
$
 ( X, \bfD _{ \alpha })
$
is b-terminal.

We will let $X$ be the singularity $\frac{1}{r}(1,1,0)$, so that the minimal discrepancy is $- 1 + 2/r$ and note that
$$
 - 1 + \frac{2}{r}
\to -1 \mbox{ as } r \to \infty.$$
The singularity
$
 X
$
can be globally presented as a toric variety using the standard lattice
$
 \bZ ^{ 3 }
$
and the cone generated by the columns of the following matrix
\[(v_1,v_2,v_3) = \begin{pmatrix}
1 & -1 & 0 \\
0 & r & 0 \\
0 & 0 & 1
\end{pmatrix}.\]
We fix a prime $p$ and we let the skew symmetric matrix corresponding to the toric Brauer class $\alpha$ be
\begin{equation}
 M _{ \alpha }
 =
 \begin{pmatrix}
0 & 0 & -r \\
0 & 0 & -1 \\
r & 1 & 0 
\end{pmatrix}
\in
\lb \bZ / p \rb ^{ 3 \times 3 }.
\end{equation}

In order to check that the pair
$
 \lb X, \bfD _{ \alpha } \rb
$
is b-terminal, it is enough to check
$
 b ( D _{ \la w \ra }; X, \bfD _{ \alpha } ) > 0
$
for any primitive vector
$
 w
$
in the cone. By using the formula \eqref{eq:formula_for_toric_b_discrepancy},
we can directly check this by elementary arguments.
\end{example}

When the dimension $n$ is odd and we have rank $n-1$, one can compute the discrepancy
for \alogsmooth Brauer pair as follows, where we do the case $n=3$ in detail.

Assume $p$ is a prime so the ramification indices are $p$ or 1. We also identify $\Z/p$ with $\mu_p^{-1}$ so that $\alpha$ is represented by a skew-symmetric matrix
\[M = \begin{pmatrix}
   0 & c_2 & - c_1 \\
   -c_2 & 0 & c_0 \\
  c_1 & -c_0 & 0 
  \end{pmatrix}\]
where $c_j \in \Z/p$. We assume that $M \neq 0$ so that it has rank 2 and $\ker M$ is the $(\Z /p)$ span of the vector  $(c_0,c_1,c_2)$. We further assume that at least two of the $c_j$ are non-zero so that $\alpha$ ramifies on all three planes. For $i = 0, \ldots, p-1, j = 0,1,2$ we let $c_{ij} \in \{ 0 ,\ldots, p-1\}$ be the smallest non-negative integer whose residue modulo $p$ is $ic_j$.

Toric exceptional divisors $E_{(a_0,a_1,a_2)}$ above $X$ correspond via the toric dictionary to primitive triples $(a_0,a_1,a_2) \in \mathbb{N}^3$. From ~\cite[Lemma~1.7b]{MR1085941}, $\alpha$ is unramified along $E_{(a_0,a_1,a_2)}$ if and only if $(a_0,a_1,a_2) \in \ker M$ modulo $p$. 

Given an integer $x$, define
\[r_p(x) = \min \{ x+yp \mid x+yp \geq 0, y \in \Z \}\] to be the least non-negative residue of $x$ modulo $p$. 
\begin{proposition}
  We use the above notation and let
  \[c = \min \{r_p(ic_0) + r_p(ic_1)+r_p(ic_2) \mid i \in (\Z/p)^* \}.\]
  Then $(X,\alpha)$ is always b-lt, but will be
\begin{enumerate}
\item b-terminal if $c > p,$ in which case $(c_0, c_1,c_2) \equiv (k,a,p-a) \pmod{p}$
  for some $k,a \not\equiv 0$ up to permutation.
 \item b-canonical if $c = p,$ in which case $\sum c_i \in p\Z$.
 \item not b-canonical if $c < p$.
\end{enumerate}
\end{proposition}
\begin{proof}
  It suffices to compute discrepancy for toric exceptional divisors $E$. This is $b'(X,\alpha;E) = \frac{1}{p}(a_0 + a_1 + a_2) - \frac{1}{r_E}$ where $r_E$ is the ramification index along $E = E_{(a_0,a_1,a_2)}$. Now $a_0 + a_1 + a_2 > 1$ so this is positive unless $e_E=1$. In this case, $(a_0,a_1,a_2) \equiv i(c_0,c_1,c_2)$ modulo $p$ for some $i$. Then $b'$ is minimized when $(a_0,a_1,a_2) =(c_{i0},c_{i1},c_{i2})$ for some $i$ whence we obtain $b'(X,\alpha;E) = \frac{1}{p}(c - p)$.  Note that the minimum occurs when $r_E=1$ and so $b'(X,\alpha;E) = b(X,\alpha;E)$, and  we obtain the first part of each statement.  To obtain the classifications in the first two cases, we use the well known characterization of toric terminal 3-fold singularities described in Example-Claim 14-2-5 \cite{MatsukiBook}.
\end{proof}
To compute the discrepancies for b-\logpairs that come from ramification information, it is necessary to compute the ramification indices globally before carrying out an \'etale localization.  Since the ramification indices change after \'etale localization, we note that the discrepancy of a Brauer pair cannot be based on local information in the sense of the following example.
\begin{example} \label{enonlocal}
Suppose that $p = 3$. We let $\alpha$ correspond to $(c_0,c_1,c_2) = (1,2,0)$ and $\alpha'$ correspond to $(c_0,c_1,c_2) = (1,1,0)$. Note that $(X,\alpha)$ is Brauer canonical but $(X,\alpha')$ is not. Furthermore, if $f:X' \to X$ is the blowup along the coordinate line $C: x_0 = x_1 = 0$, then the discrepancy of $(X,\alpha')$ along the exceptional divisor is negative. However, if $P\in C$ is a general point, then  $(X,\alpha)$ and $(X,\alpha')$ are isomorphic in an \'etale neighbourhood of $P$.  
\end{example}

\bibliographystyle{skalpha}
\bibliography{mainbibs}

\end{document}